\documentclass[11pt,a4paper]{amsart}
 
\usepackage[utf8]{inputenc}
\usepackage{amsmath,amssymb,amsthm,mathtools}
\usepackage{geometry}
\usepackage{graphicx}
\usepackage{xcolor}
\usepackage{hyperref}
\usepackage{mathrsfs}

\newtheorem{theorem}{Theorem}[section]
\newtheorem{lemma}[theorem]{Lemma}
\newtheorem{proposition}[theorem]{Proposition}
\newtheorem{corollary}[theorem]{Corollary}
\theoremstyle{definition}
\newtheorem{definition}[theorem]{Definition}
\newtheorem{assumption}[theorem]{Assumption}
\newtheorem{remark}[theorem]{Remark}

\newcommand{\R}{\mathbb{R}}
\newcommand{\Gammah}{\Gamma_h}
\newcommand{\nablag}{\nabla_{\Gamma}}
\newcommand{\nablagh}{\nabla_{\Gamma_h}}
\newcommand{\deltag}{\Delta_{\Gamma}}
\newcommand{\deltagh}{\Delta_{\Gamma_h}}
\newcommand{\tdeltagh}{\widetilde{\Delta}_{\Gamma_h}}
\newcommand{\Eh}{\mathcal E_h}
\newcommand{\Kh}{\mathcal K_h}
\newcommand{\tr}{\operatorname{tr}}
\newcommand{\di}{\operatorname{div}}
\newcommand{\diam}{\operatorname{diam}}
\newcommand{\vn}{\mathbf{n}}

\newcommand{\vx}{\mathbf{x}}

\newcommand{\vnu}{\boldsymbol{\nu}}

\newcommand{\mean}[1]{\{\!\{#1\}\!\}}
\newcommand{\norm}[1]{\lVert#1\rVert}
\newcommand{\enorm}[1]{\lvert\!\lvert\!\lvert #1\rvert\!\rvert\!\rvert_h}

\title[Reconstructed-Laplacian method on surfaces]{A Reconstructed-Laplacian Method for the Surface Biharmonic Equation on Parametric Meshes}

\author{Shuo Yang}
\thanks{Beijing Institute of Mathematical Sciences and Applications, Beijing, 101408, China}
\email{shuoyang@bimsa.cn}

\date{\today}   

\begin{document}

\begin{abstract}
We develop and analyze a continuous/discontinuous Galerkin (CDG) method based on reconstructed surface Laplacians for the biharmonic equation on a smooth closed surface.  
Continuous mapped finite elements of degree $k\ge2$ are used on fitted parametric meshes of degree $r\ge1$, while a discontinuous degree-$(k-2)$ lifting corrects the broken Laplace-Beltrami operator for two-sided conormal-flux jumps.  
The resulting completed-square form is coercive on the mean-zero space for every fixed $\beta>0$, without requiring a sufficiently large penalty parameter.    
Under the standard geometric assumptions, we prove that the energy and reconstructed-Laplacian errors are $\mathcal O(h^{k-1}+h^r)$, and the $L^2$-error is $\mathcal O(h^{q_k}+h^{r+1})$, where $q_2=2$ and $q_k=k+1$ for $k\ge3$.
Benchmark computations support these rates, while a surface Swift-Hohenberg experiment illustrates the extension of the method to nonlinear Laplacian-dominated models.
\end{abstract}

\maketitle

%=======================================================================
\section{Introduction}
%=======================================================================
Fourth-order partial differential equations on surfaces occur in phase-field models, surface diffusion, membrane and shell theories, geometric regularization, and pattern formation; representative examples include the surface Cahn-Hilliard equation, fourth-order membrane models, the surface Kirchhoff plate equation, and surface Swift-Hohenberg equations \cite{du2011cahn,elliott2015cahn,stinner2020fourth,walker2022kirchhoff,yu2023stabilized}. 
As a canonical elliptic model, let $\Gamma\subset\mathbb R^3$ be a smooth, connected, orientable, closed surface and consider the surface biharmonic equation 
\[
  \deltag^2u=f\quad\text{on }\Gamma,
  \qquad
  (u,1)_\Gamma=(f,1)_\Gamma=0.
\]
Its variational formulation is $H^2$-elliptic.  A conforming discretization requires a globally $C^1$ space, while standard fitted surface meshes are only $C^0$ across element interfaces.  
In addition to this classical fourth-order difficulty, a surface method must control variational crimes in geometric approximations.   
These effects become especially important when both the finite element and the surface geometry are approximated at high order.  
We refer to \cite{dziuk2013finite,bonito2020finite} for general surface finite element theory and to \cite{demlow2009higher} for higher-order parametric surface approximation.

Several non-conforming finite element methods have been designed and analyzed for the surface biharmonic problem in the past decade.  
Larsson and Larson \cite{larsson2017continuous} introduced a $C^0$ interior penalty discontinuous Galerkin (IPDG) method on piecewise planar surfaces and proved energy- and $L^2$-error estimates.  
More recent alternatives include a stabilized surface nonconforming method based on a New-Zienkiewicz-type element \cite{wu2025stabilized}, a continuous linear method based on surface gradient recovery \cite{cai2026continuous}, and an unfitted $C^0$ interior-penalty TraceFEM \cite{neilan2025tracefem}.  
Related fourth-order surface formulations include second-order splitting methods studied in \cite{elliott2019splitting}.
Also, the surface Hellan-Herrmann-Johnson method was proposed for the Kirchhoff plate equation on surfaces in \cite{walker2022kirchhoff}, and numerical methods for surface Cahn-Hilliard-type problems were investigated in \cite{du2011cahn,elliott2015cahn,stinner2020fourth}.  
These approaches avoid globally $C^1$ surface elements in different ways, for example by introducing auxiliary unknowns, weakly enforcing interelement derivative continuity, or recovering higher derivatives.

The present construction is also connected with the broader discontinuous Galerkin and local discontinuous Galerkin (LDG) literature.  
On planar domains, the LDG method was introduced for convection-diffusion systems in \cite{cockburn1998local}; its elliptic error analysis and its relation to other DG formulations were developed in \cite{castillo2000apriori,arnold2002unified}.  Extensions to equations with higher-order derivatives were studied in \cite{yan2002local}, while $C^0$ interior-penalty methods for fourth-order elliptic problems were developed in \cite{engel2002continuous,brenner2005c0}.  
More recently, lifting-based reconstructed Hessians have been used for nonlinear plate models \cite{bonito2022ldg,bonito2023numerical,bonito2024gamma,bonito2024finite}.  On surfaces, DG methods for second-order elliptic equations were analyzed for affine and higher-order geometry in \cite{dedner2013analysis,antonietti2014high}; see also the adaptive method \cite{dedner2016adaptive} and the hybridizable DG method \cite{cockburn2016hybridizable}.  
A surface LDG scheme for the Cahn-Hilliard equation was proposed in
\cite{xu2024surfaceLDG}, where the authors rewrite the equation as a
first-order system and couple fully discontinuous scalar and vector
auxiliary variables through numerical fluxes.

Our method is LDG-inspired but differs in formulation from a fully discontinuous first-order system as in its original emergence. 
Instead, we construct a reconstructed surface Laplacian operator and replace the original Laplace-Beltrami operator in the variational formulation to design the discretization; this generalizes the ideas of \cite{bonito2022ldg,bonito2023numerical} from planar domains to surfaces. 
A primary theoretical objective of this paper is to investigate the interplay between the concept of reconstructed operators and geometric variational crimes in surface PDEs, as well as their subsequent impact on the numerical analysis.
    
The discrete solution is in an $H^1$-conforming mapped space $W_h$ of degree $k$, whereas a scalar lifting takes values in a discontinuous space $Q_h$ of degree $k-2$.  Let $J(v)$ denote the sum of the two outward discrete conormal derivatives on an interior edge. 
The lifting $R_hJ(v)$ is defined by duality against edge averages, and the reconstructed surface Laplacian is
\[
  \tdeltagh v:=\deltagh v-R_hJ(v).
\]
The discrete bilinear form is
\[
  (\tdeltagh v_h,\tdeltagh w_h)_{\Gamma_h}
  +\beta\sum_{E\in\mathcal E_h}h_E^{-1}
  (J_E(v_h),J_E(w_h))_E.
\]
Thus the consistency terms that appear separately in an IPDG formulation are incorporated into the terms like $(\deltagh v_h,R_hJ(w_h))_{\Gamma_h}$ for affine meshes.  
The additional lifting-lifting term completes this square and is the mechanism behind stability of this method for \emph{every} fixed $\beta>0$.   
Standard symmetric IPDG, by contrast, generally requires the stabilization parameter $\beta$ to exceed a threshold.  

There is also a geometric reason to reconstruct the scalar Laplace-Beltrami operator rather than a surface Hessian.  
Let $\nabla_\Gamma^2$ denote the covariant tangential Hessian and $\mathcal K$ the Gaussian curvature, the integrated Bochner identity on a closed surface is
\begin{equation}\label{eq:intro-bochner}
  (\deltag v,\deltag w)_\Gamma
  =(\nabla_\Gamma^2v,\nabla_\Gamma^2w)_\Gamma
  +(\mathcal K\nablag v,\nablag w)_\Gamma.
\end{equation}
Consequently, a Hessian reconstruction alone does not reproduce the biharmonic energy unless the curvature term is included.  
Reconstructing $\deltag$ directly is sufficient and only resorts to a single scalar lifting.

The analysis is carried out for arbitrary fixed solution degree $k\ge2$ and
geometry degree $r\ge1$.  The computational surface is generated by a general
degree-$r$ parametric curving map; it is not assumed to be the exact
closest-point interpolant.  The closest-point projection is used as an
analytical lift, while explicit shape-regularity, distance, normal, measure,
and conormal estimates are imposed on the actual computational maps.  This
separation is important in implementations based on practical high-order mesh
generators.  

Regarding the interplay between the reconstructed surface Laplacian and high-order geometric approximation, a key observation is that, on a curved element, $\deltagh W_h$ is generally not contained in $Q_h$. Consequently, an associated $L^2$ projection must be introduced and tracked in the analysis, together with the nonzero geometric jump of the closest-point extension of a smooth function.

The main contributions of the paper and features of the method are as follows.
\begin{itemize}
\item \textit{Reconstructed surface Laplacian.} We formulate a CDG method based on a reconstructed surface Laplacian operator on arbitrary-order fitted parametric meshes. The formulation treats the two element-side conormals independently and remains valid when element maps, Laplacians, and flux jumps are nonpolynomial on curved meshes.

\item \textit{Coercivity without penalty tuning.} We prove continuity and coercivity in energy norm for every fixed $\beta>0$.  The method therefore avoids selecting a penalty above an unknown stability threshold.  

\item \textit{General-order error estimates.} We separate approximation, projection, and geometric consistency errors for general $k$ and $r$. The energy error and the error in the reconstructed Laplacian are bounded by $\mathcal O(h^{k-1}+h^r)$.  A duality argument, including a higher-order smooth-smooth geometric consistency estimate, gives $\mathcal O(h^{q_k}+h^{r+1})$ in $L^2$-error, where $q_2=2$ and $q_k=k+1$ for $k\ge3$.

\item \textit{Curved-mesh assembly.} We give a general assembly formula that works on generally curved elements.  Benchmark computations support the theoretical rates, and direct geometric diagnostics provide numerical evidence for the assumed high-order geometry estimates with the computational mesh used practically.

\item \textit{Extension to nonlinear surface models.} 
The proposed method is conceptually straightforward, offering strong potential for generalization to other fourth-order surface PDEs. Specifically, substituting $\deltag$ with $\tdeltagh$ allows for the direct derivation of discretized formulations, thereby eliminating the need for cumbersome terms or non-standard treatments.
We illustrate that the same reconstructed Laplacian can be reused in nonlinear surface models with a Laplacian-dominated principal part.  
A surface Swift-Hohenberg experiment demonstrates this extension.
\end{itemize}

The remainder of the paper is organized as follows.  Section~2 introduces the parametric surface geometry and its approximation properties.  
Section~3 defines the finite element spaces, lifting, reconstructed surface Laplacian, and discrete problem and establishes stability.  
Section~4 proves the energy, reconstructed-Laplacian, and $L^2$ error estimates.  Section~5 discusses assembly on curved meshes and presents benchmark numerical experiments.
Section~6 treats Laplacian-dominated nonlinear problems and the surface Swift-Hohenberg example.  
Conclusions are given in Section~7, and technical geometric estimates are collected in the appendix.

%=======================================================================
\section{Geometric approximations}
%=======================================================================
\subsection{Tangential calculus and the continuous problem}

Let $\Gamma\subset\R^3$ be a connected, compact, orientable surface without boundary, with unit normal $\vn$.    
The orthogonal projection onto the tangent space of $\Gamma$ is
\begin{equation}\label{def:tangent-proj}
  P:=I-\vn\otimes\vn.
\end{equation}
For $v:\Gamma\to\R$ and $X:\Gamma\to\R^3$, let $\widetilde v, \widetilde X$ be any sufficiently smooth extension to a neighborhood of $\Gamma$. The tangential gradient, surface divergence, and Laplace-Beltrami operator are
\begin{equation}\label{def:tangent-op}
  \nablag v:=P\nabla\widetilde v,
  \qquad
  \di_{\Gamma}X:=\tr(P\nabla\widetilde X),
  \qquad
  \deltag v:=\di_{\Gamma}(\nablag v).
\end{equation}
These definitions are independent of the chosen extensions.  We use the standard Sobolev spaces on $\Gamma$; see
\cite{dziuk2013finite,demlow2009higher,bonito2020finite} for a more detailed introduction of the tangential calculus.

We first consider the surface biharmonic equation 
\begin{equation}\label{eq:strong-problem}
  \deltag^2u=f\quad\text{on }\Gamma,
  \qquad
  (u,1)_\Gamma=0,
\end{equation}
under the compatibility condition
\begin{equation}\label{eq:compatibility}
  (f,1)_\Gamma=0.
\end{equation}
For $m\geq0$, define 
\begin{equation}\label{eq:mean-zero-space}
 H^m_\#(\Gamma):=\left\{v\in H^m(\Gamma):(v,1)_\Gamma=0\right\}.
\end{equation}
With a given $f\in L^2_{\#}(\Gamma)$, the weak formulation is to find $u\in H^2_\#(\Gamma)$ such that
\begin{equation}\label{eq:continuous-weak}
  a(u,v):=(\deltag u,\deltag v)_\Gamma=(f,v)_\Gamma
  \qquad\forall v\in H^2_\#(\Gamma).
\end{equation}
On a connected closed surface, the kernel of $\deltag$ consists of constants, and thus $\norm{\deltag v}_{L^2(\Gamma)}$ is equivalent to $\norm{v}_{H^2(\Gamma)}$ on $H^2_\#(\Gamma)$.  
Hence \eqref{eq:continuous-weak} is well-posed by the Lax-Milgram theorem. 
For an integer $s\ge4$, assume that $\Gamma$ is a closed, connected surface of class $C^{s-1,1}$. Then, for every $f\in H^{s-4}_\#(\Gamma)$, the problem \eqref{eq:strong-problem} has a unique solution in $ H^s_\#(\Gamma)$ and
\begin{equation}\label{eq:higher-regularity}
  \norm{u}_{H^s(\Gamma)}
  \le
  C\norm{f}_{H^{s-4}(\Gamma)}.
\end{equation}
This follows from \cite[Section~4.4]{benavides2026lp}.
For smooth surfaces, the corresponding general elliptic estimate can be found in, for instance \cite{besse2007einstein}.

\subsection{Arbitrary-order parametric surface meshes}
\label{subsec:parametric-mesh}
We fix an integer $r\ge1$ to be the degree of the geometric approximation.  
Indeed, if $\Gamma\in C^q$, $q\ge2$, then the signed distance function $d$ belongs to $C^q$ in a proper tubular neighborhood, while the closest-point projection $p$ belongs to $C^{q-1}$; see \cite{gilbarg1998elliptic,bonito2020finite}.
To ensure the regularity of the closest-point lift used in the analysis and to include the degree-$r$ closest-point interpolant in Remark~\ref{rem:admissible-geometric-constructions} as one admissible geometric construction, we require $\Gamma\in C^{r+2}$.
In the subsequent analysis, with a given finite element approximation degree $k\ge2$, we need the elliptic regularity \eqref{eq:higher-regularity} for $s=s_k=\max\{k+1,4\}$. 
We consider $\Gamma\in C^{r+2}\cap C^{s_k,1}$ in the rest of the paper.  The $C^{s_k-1,1}$ regularity in \eqref{eq:higher-regularity} is sufficient for the PDE estimate itself; the additional derivative is a convenient sufficient condition for the $W^{s_k,\infty}$ regularity of the closest-point pullbacks used in the high-order norm equivalences and interpolation estimates below.
We note that this condition is stronger than necessary for some individual estimates, but it ensures that the signed distance function $d$, the closest-point projection $p$, and their derivatives used below are well-defined and uniformly bounded, and that the elliptic regularity~\eqref{eq:higher-regularity} always holds.

For a sufficiently small $\delta_0>0$, we define the tubular neighborhood of $\Gamma$ as  
\begin{equation}\label{def:tub-neigh}
  U_{\delta_0}(\Gamma):=\{\vx\in\R^3:|d(\vx)|<\delta_0\},
\end{equation}
and there exists a unique closest-point projection $p$ onto $\Gamma$ in $U_{\delta_0}$.  
In this neighborhood, we have 
\begin{equation}\label{eq:closest-point}
  \vn=\nabla d,
  \qquad
  \mathcal W=D^2d,
  \qquad
  p(\vx)=\vx-d(\vx)\vn(\vx),
  \qquad
  \nabla p=P-d\mathcal W,
\end{equation}
where the first two equations mean that $\nabla d$ and $D^2d$ are proper extensions of $\vn$ and $\mathcal W$ in $U_{\delta_0}$. 
Here $\mathcal W$ is the Weingarten map of the surface, with the sign determined by that of $d$.  
The existence, regularity, and uniform boundedness of these quantities follow from the regularity of $\Gamma$; see \cite{dziuk2013finite,demlow2009higher,bonito2020finite}.

We first consider a first order geometric approximation of $\Gamma$. 
Let
\begin{equation*}
  \gamma_h=\bigcup_{T\in\mathcal T_h}T
  \subset U_{\delta_0}(\Gamma)
\end{equation*}
be a connected polyhedral surface whose vertices lie on $\Gamma$.  We assume that $\{\mathcal T_h\}_h$ is a shape-regular family of conforming affine triangulations.  
We set
\begin{equation*}
  h_{T}:=\diam(T),
  \qquad
  h_{e}:=\diam(e),
  \qquad
  h:=\max_{T\in\mathcal T_h}h_{T}.
\end{equation*}
Here $T\in\mathcal T_h$, $e\in\mathscr E_h$, and $\mathscr E_h$ denotes the set of edges of the parent triangulation $\mathcal T_h$. 
For $e:=\partial T^{+}\cap\partial T^{-}$, by shape regularity there holds
\begin{equation*}
  h_{e}\simeq h_{T^{+}}\simeq h_{T^{-}},
\end{equation*}
with constants independent of the local mesh sizes.
For sufficiently small $h$, $p:\gamma_h\to\Gamma$ is a bijection.  
Let $F_T:\widehat T\to T$ be the affine map from the reference triangle. Note that $c_0h_T|\xi|\le |DF_T\xi|\le C_0h_T|\xi|$ for any $\xi\in\mathbb R^2.$

In order to define the computational mesh with general geometric order $r\geq1$, we consider a general globally continuous degree-$r$ map $p_h^r:\gamma_h\to U_{\delta_0}(\Gamma)$ such that 
\begin{equation}\label{eq:computational-curving-map}
  (p_h^r|_T)\circ F_T\in[\mathcal P_r(\widehat T)]^3
  \qquad\forall T\in\mathcal T_h.
\end{equation}
In particular, for $r=1$, we set $p_h^1=\operatorname{id}_{\gamma_h}$.  For a given $T\in\mathcal T_h$, define
\begin{equation}\label{eq:curved-element-map}
  K:=p_h^r(T),
  \qquad
  F_K^r:=(p_h^r|_T)\circ F_T,
  \qquad
  \Gamma_h:=p_h^r(\gamma_h).
\end{equation}
The global continuity of $p_h^r$ implies that
\begin{equation}\label{eq:curved-mesh}
  \Kh:=\{K=p_h^r(T):T\in\mathcal T_h\},
  \qquad
  \Eh:=\{E=p_h^r(e):e\in \mathscr E_h\}
\end{equation}
form a conforming, generally only $C^0$, curved triangulation.

\begin{assumption}[Shape regular curved meshes] 
\label{ass:computational-element-maps}
For all sufficiently small $h$, every $F_K^r$ is injective, has rank two, and satisfies
\begin{equation}\label{eq:parametric-shape-regularity}
  c h_T^2|\xi|^2
  \le |DF_K^r(\widehat x)\xi|^2
  \le C h_T^2|\xi|^2,
  \qquad
  \widehat x\in\widehat T,\quad \xi\in\mathbb R^2,
\end{equation}
as well as the scaled higher-derivative bounds
\begin{equation}\label{eq:curved-map-derivative-bounds}
  \norm{D^mF_K^r}_{L^\infty(\widehat T)}
  \le C h_T^m,
  \qquad 2\le m\le r,
\end{equation}
and the uniform bi-Lipschitz estimate
\begin{equation}\label{eq:curved-map-bilipschitz}
  c h_T|\widehat x-\widehat y|
  \le |F_K^r(\widehat x)-F_K^r(\widehat y)|
  \le C h_T|\widehat x-\widehat y|,
  \qquad\widehat x,\widehat y\in\widehat T.
\end{equation}
The constants are independent of $T$ and $h$, the element orientations are consistent, and the closest-point projection $p:\Gamma_h\to\Gamma$ is a bijection.
\end{assumption}
 
This assumption asserts that the family of curved meshes is shape-regular.  
The bi-Lipschitz estimate implies
\begin{equation*}
  h_K:=\diam(K)\simeq h_T.
\end{equation*}
Moreover, if $E=\partial K^+\cap\partial K^-$, the shape regularity of the affine parent mesh and the uniform bi-Lipschitz bounds also lead to
\begin{equation*}
  h_E:=\diam(E)
  \simeq h_{K^+}
  \simeq h_{K^-},
\end{equation*}
with constants uniform over the mesh family. With a slight abuse of notation, we now redefine $h:=\max_{E\in\mathcal E_h}h_{E}$.

\begin{remark}[Admissible geometric constructions]
\label{rem:admissible-geometric-constructions}
The closest-point interpolant is one admissible construction, and it is used in \cite{demlow2009higher} to construct the computational mesh. 
Namely, if $I_T^r$ denotes the degree-$r$ nodal Lagrange interpolant, one may take
\begin{equation}\label{eq:parametric-closest-point-map}
  p_h^r|_T:=I_T^r(p|_T).
\end{equation}
The agreement of the nodal values on common parent edges gives global continuity. The standard interpolation estimates, geometric approximation estimate $\norm{p-id}_{L^\infty(T)}\lesssim h_T^2$, the Bramble-Hilbert lemma on $\widehat T$ and the standard scaling argument ensure that Assumption~\ref{ass:computational-element-maps} holds for $p_h^r$ defined in \eqref{eq:parametric-closest-point-map}.

However, in general, one may want to avoid using the closest-point projection $p$ to construct the computational mesh in practice, because obtaining an explicit expression for $p$ on general surfaces $\Gamma$ is challenging (if not impossible), even though $p$ is convenient for numerical analysis; see \cite{bonito2020finite} for further discussion.
\end{remark}

\subsection{Discrete differential geometry and transformations}

On $K=F_K^r(\widehat T)$, we define the oriented discrete unit outer normal and
tangential projection to $\Gammah$ by
\begin{equation}\label{eq:curved-discrete-normal}
  \vn_h\circ F_K^r
  :=\frac{\partial_1F_K^r\times\partial_2F_K^r}
  {|\partial_1F_K^r\times\partial_2F_K^r|},
  \qquad
  P_h:=I-\vn_h\otimes\vn_h.
\end{equation}
They are smooth in the interior of each element and generally discontinuous across an edge. 
Tangential differential operators on $\Gammah$ are defined elementwise; in particular,
\begin{equation*}
  \nablagh v=P_h\nabla\widetilde v,
  \qquad
  \deltagh v=\di_{\Gammah}(\nablagh v).
\end{equation*}

For $v$ on $\Gamma$, its closest-point extension is $v^e:=v\circ p$.  For
$v_h$ on $\Gammah$, its lift is $v_h^l:=v_h\circ p^{-1}$.  Exact geometric
quantities evaluated at $p(x)$ are also denoted by a superscript $e$.  The
exact lifted mesh is
\begin{equation}\label{eq:lifted-mesh}
  \mathcal K_h^l:=\{K^l=p(K):K\in\Kh\},
  \qquad
  \mathcal E_h^l:=\{E^l=p(E):E\in\Eh\}.
\end{equation}

Define the surface and edge measure quotients by
\begin{equation}\label{eq:measure-quotients}
  ds=\mu_h\,ds_h,
  \qquad
  ds_{E^l}=\mu_{h,E}\,ds_E.
\end{equation}
The tangential derivative of the closest-point projection, restricted to
$T_x\Gammah$, is represented by
\begin{equation}\label{eq:B-map}
  B_h(x):=P^e(I-d\mathcal W)P_h
  :T_x\Gammah\longrightarrow T_{p(x)}\Gamma.
\end{equation}
For sufficiently small $h$, $B_h$ is invertible between these tangent spaces
and
\begin{equation}\label{eq:gradient-transform}
  \nablagh v^e=B_h^T(\nablag v)^e,
  \qquad
  (\nablag v_h^l)^e=B_h^{-T}\nablagh v_h.
\end{equation}

\subsection{Geometric estimates}
We state explicitly the geometric approximation properties required of the actual computational maps on the generated mesh family.

\begin{assumption}[Parametric geometry estimates]
\label{lem:parametric-geometry}
In addition to Assumption~\ref{ass:computational-element-maps}, the actual curving map satisfies, for every $T\in\mathcal T_h$,
\begin{equation}\label{eq:scaled-distance-defect}
  \norm{d\circ p_h^r}_{W^{m,\infty}(T)}
  \le Ch_T^{r+1-m},
  \qquad m=0,1,2.
\end{equation}
Moreover, for $h$ sufficiently small,
\begin{align}
  \norm{d}_{L^\infty(\Gamma_h)}
  &\le Ch^{r+1},
  &
  \norm{\vn^e-\vn_h}_{L^\infty(\Gamma_h)}
  &\le Ch^r,
  \label{eq:dist-vel}
  \\
  \norm{1-\mu_h}_{L^\infty(\Gamma_h)}
  &\le Ch^{r+1},
  &
  \norm{1-\mu_{h,E}}_{L^\infty(\mathcal E_h)}
  &\le Ch^{r+1},
  \label{eq:mu-errors}
\end{align}
Furthermore,  
\begin{equation}\label{eq:Bh-bounded}
  \norm{B_h}_{L^\infty(\Gamma_h)}
  +\norm{B_h^{-1}}_{L^\infty(\Gamma_h)}
  \le C, 
\end{equation}
\begin{equation}\label{eq:PhBh-est}
  \norm{P^e-B_hB_h^T}_{L^\infty(\Gamma_h)}
  +
  \norm{P_h-B_h^TB_h}_{L^\infty(\Gamma_h)}
  \le Ch^{r+1}.
\end{equation}
The corresponding elementwise and edgewise estimates hold with $h$ replaced by $h_K$ and $h_E$, respectively.
\end{assumption}

For the particular closest-point interpolant \eqref{eq:parametric-closest-point-map}, the distance and normal estimates \eqref{eq:scaled-distance-defect} and \eqref{eq:dist-vel} follow from \cite[Proposition~2.3]{demlow2009higher}; see also \cite[Section~4.2 and Lemma~4.4]{dziuk2013finite}.
The surface and edge measure estimates \eqref{eq:mu-errors} are proved in \cite[Lemma~4.1]{antonietti2014high}.
The remaining estimates follow naturally from the gradient relations \cite[equations~(2.11)-(2.13)]{demlow2009higher} and the estimate \cite[Proposition~4.1]{demlow2009higher}; compare also to \cite[Lemmas~20-23]{bonito2020finite}, \cite[Lemma~4.4]{dziuk2013finite} and \cite[Lemma~4.1]{antonietti2014high}.
For a general practical curving element map that is used to construct the computational mesh, Assumption~\ref{lem:parametric-geometry} could be assessed numerically for the mesh family used in the experiments. 

The following standard norm equivalence results and scaled trace inequality will also be used repeatedly in later numerical analysis. 
We refer to \cite{dziuk2013finite,antonietti2014high,larsson2017continuous,bonito2020finite,elliott2021unified} for their proofs and detailed discussions.

\begin{lemma}[Norm equivalence]
\label{lem:norm-trace}
For $v\in H^1(\Gamma)$ and $v_h\in H^1(\Gamma_h)$, the uniform $L^2$- and $H^1$-norm equivalences are valid as follows:
\begin{equation}\label{eq:extension-norm-equivalence}
\begin{aligned}
  \norm{v}_{L^2(\Gamma)}&\simeq\norm{v^e}_{L^2(\Gammah)},
  &
  \norm{\nablag v}_{L^2(\Gamma)}
  &\simeq\norm{\nablagh v^e}_{L^2(\Gammah)},
  \\
  \norm{v_h^l}_{L^2(\Gamma)}&\simeq\norm{v_h}_{L^2(\Gammah)},
  &
  \norm{\nablag v_h^l}_{L^2(\Gamma)}
  &\simeq\norm{\nablagh v_h}_{L^2(\Gammah)}.
\end{aligned}
\end{equation}
For an integer $2\le m\le s_k$ fixed independently of $h$, let
$v\in H^m(\Gamma)$ and let $v_h$ be elementwise $H^m$ on $\Gamma_h$.
We define the broken seminorm 
\[
  |w|_{H^j(\mathcal K_h)}^2
  :=
  \sum_{K\in\mathcal K_h}
  \norm{D_{\Gamma_h}^j w}_{L^2(K)}^2,
  \qquad
  |z|_{H^j(\mathcal K_h^l)}^2
  :=
  \sum_{K^l\in\mathcal K_h^l}
  \norm{D_\Gamma^j z}_{L^2(K^l)}^2.
\]
Here $D_\Gamma^j$ and $D_{\Gamma_h}^j$ denote the corresponding tangential derivatives, taken elementwise on the broken meshes, and
$D_\Gamma^0=D_{\Gamma_h}^0=I$.
Then
\begin{equation}\label{eq:higher-order-norm-equivalence}
\begin{aligned}
  |v^e|_{H^m(\mathcal K_h)}
  &\le C\left(
    \sum_{j=1}^m |v|_{H^j(\Gamma)}^2
  \right)^{1/2},
  &
  |v|_{H^m(\Gamma)}
  &\le C\left(
    \sum_{j=1}^m |v^e|_{H^j(\mathcal K_h)}^2
  \right)^{1/2},
  \\
  |v_h^l|_{H^m(\mathcal K_h^l)}
  &\le C\left(
    \sum_{j=1}^m |v_h|_{H^j(\mathcal K_h)}^2
  \right)^{1/2},
  &
  |v_h|_{H^m(\mathcal K_h)}
  &\le C\left(
    \sum_{j=1}^m |v_h^l|_{H^j(\mathcal K_h^l)}^2
  \right)^{1/2}.
\end{aligned}
\end{equation}
The constant $C$ is independent of $h$. Consequently, with the definitions
\begin{equation*}
  \norm{w}_{H^m(\mathcal K_h)}^2
  :=
  \sum_{K\in\mathcal K_h}\sum_{j=0}^{m}
    \norm{D_{\Gamma_h}^{\,j}w}_{L^2(K)}^2,
\quad
  \norm{z}_{H^m(\mathcal K_h^l)}^2
  :=
  \sum_{K^l\in\mathcal K_h^l}\sum_{j=0}^{m}
    \norm{D_{\Gamma}^{\,j}z}_{L^2(K^l)}^2,
\end{equation*}
one has the uniform full-norm equivalences
\begin{equation}\label{eq:broken-full-norm-equivalence}
  \norm{v^e}_{H^m(\mathcal K_h)}
  \simeq
  \norm{v}_{H^m(\Gamma)},
  \qquad
  \norm{v_h^l}_{H^m(\mathcal K_h^l)}
  \simeq
  \norm{v_h}_{H^m(\mathcal K_h)}.
\end{equation}
\end{lemma}

\begin{lemma}[Scaled trace inequalities]
For $z\in H^1(K)$ and $E\subset\partial K$,
\begin{equation}\label{eq:scaled-trace}
  \norm{z}_{L^2(E)}^2
  \le C\left(
    h_K^{-1}\norm{z}_{L^2(K)}^2
    +h_K\norm{\nablagh z}_{L^2(K)}^2
  \right).
\end{equation}
The same estimate holds on $(K^l,E^l)$ with $\nablag$ in place of $\nablagh$. 
\end{lemma}

\subsection{Curved-edge conormals}
Let $E=\partial K^+\cap\partial K^-$.  Denote the outward unit conormals, tangent to $K^\pm$, by $\vnu_E^\pm$.  Let $\vnu_{E^l}^\pm$ be the corresponding exact conormals on the lifted edge $E^l=p(E)$.  
Since the lifted elements belong to the same smooth surface, $\vnu_{E^l}^++\vnu_{E^l}^-=0$.
On the other hand, the discrete surface is only $C^0$ and, in general, $\vnu_E^++\vnu_E^-\ne0$.

The estimates below are the higher-order counterparts of the conormal
estimates in Lemma~3.1 of \cite{larsson2017continuous}.  

\begin{lemma}[Conormal geometry]
\label{lem:conormal-geometry}
For any $E\in\Eh$ and $h$ sufficiently small,
\begin{align}
  \norm{(\vnu_{E^l}^{\,\pm})^e-P^e\vnu_E^\pm}_{L^\infty(E)}
  &\le Ch^{r+1},
  \label{eq:individual-conormal-comparison}
  \\
  \norm{\vnu_E^++\vnu_E^-}_{L^\infty(E)}
  &\le Ch^r,
  \label{eq:ambient-conormal-sum}
  \\
  \norm{P^e(\vnu_E^++\vnu_E^-)}_{L^\infty(E)}
  &\le Ch^{r+1}.
  \label{eq:tangential-conormal-sum}
\end{align}
Consequently, for $v\in H^2(\Gamma)$,
\begin{equation}\label{eq:local-smooth-geometric-jump}
  \norm{J_E(v^e)}_{L^2(E)}
  \le Ch^{r+1}\norm{(\nablag v)^e}_{L^2(E)},
\end{equation}
where
\begin{equation}\label{eq:JE-smooth}
  J_E(v^e)
  :=\vnu_E^+\cdot\nablagh v^e|_{K^+}
    +\vnu_E^-\cdot\nablagh v^e|_{K^-}.
\end{equation}
The global estimate is
\begin{equation}\label{eq:smooth-geometric-jump-global}
  \left(
    \sum_{E\in\Eh}h_E^{-1}
    \norm{J_E(v^e)}_{L^2(E)}^2
  \right)^{1/2}
  \le Ch^r\norm{v}_{H^2(\Gamma)}.
\end{equation}
\end{lemma}
 
\begin{proof}
The conormal estimate \eqref{eq:individual-conormal-comparison} follows the argument of \cite[Lemma~4.1]{antonietti2014high}, now applied under Assumption \ref{lem:parametric-geometry} (Parametric geometry estimates). 
Using \eqref{eq:individual-conormal-comparison} and the fact that $(\vnu_{E^l}^{\,+})^e+(\vnu_{E^l}^{\,-})^e=0$, we obtain 
\[
  \norm{P^e(\vnu_E^++\vnu_E^-)}_{L^\infty(E)}
  \le Ch^{r+1}.
\]

Furthermore, due to the fact that $\vn_h^\pm\cdot\vnu_E^\pm=0$ and estimate \eqref{eq:dist-vel},
\[
  |\vn^e\cdot\vnu_E^\pm|
  =
  |(\vn^e-\vn_h^\pm)\cdot\vnu_E^\pm|
  \le Ch^r.
\]
Decomposing $\vnu_E^++\vnu_E^-$ into its exact tangential and normal
components therefore gives
\[
  \norm{\vnu_E^++\vnu_E^-}_{L^\infty(E)}
  \le Ch^r.
\]
This proves \eqref{eq:ambient-conormal-sum}.

The closest-point gradient transformation gives $\nablagh v^e=P_h^\pm(P^e-d\mathcal W)(\nablag v)^e$ on $K^\pm$.  Since $\vnu_E^\pm$ is tangent to $K^\pm$ and $(P^e-d\mathcal W)(\nablag v)^e$ is tangential to $\Gamma$ at $p(x)$, it follows that

\[
\begin{aligned}
  J_E(v^e)
  &=(\vnu_E^++\vnu_E^-)\cdot(P^e-d\mathcal W)(\nablag v)^e=P^e(\vnu_E^++\vnu_E^-)\cdot(P^e-d\mathcal W)(\nablag v)^e.
\end{aligned}
\]
The uniform boundedness of $P^e$ and $d\mathcal W$ and \eqref{eq:tangential-conormal-sum} now imply
\[
  \norm{J_E(v^e)}_{L^2(E)}
  \le Ch^{r+1}\norm{(\nablag v)^e}_{L^2(E)},
\]
which is \eqref{eq:local-smooth-geometric-jump}.

Finally, using the corresponding local estimate with $h$ replaced by $h_E$, followed by the scaled trace inequality \eqref{eq:scaled-trace}, gives
\[
\begin{aligned}
  \sum_{E\in\Eh}h_E^{-1}\norm{J_E(v^e)}_{L^2(E)}^2
  &\le
  C\sum_{E\in\Eh}h_E^{2r+1}\norm{(\nablag v)^e}_{L^2(E)}^2
  \\
  &\le
  C\sum_{K\in\Kh}
  \left(
    h_K^{2r}\norm{(\nablag v)^e}_{L^2(K)}^2
    +h_K^{2r+2}\norm{\nablagh(\nablag v)^e}_{L^2(K)}^2
  \right)
  \\
  &\le Ch^{2r}\norm{v}_{H^2(\Gamma)}^2,
\end{aligned}
\]
where shape regularity, finite overlap, and the norm equivalences in Lemma \ref{lem:norm-trace} have been used, and this leads to \eqref{eq:smooth-geometric-jump-global}.
\end{proof}

%=======================================================================
\section{Finite element space and discretization}
%=======================================================================

\subsection{Mapped spaces and the conormal-flux jump}

Fix a polynomial degree $k\ge2$ and set $\ell:=k-2$. 
The geometric approximation degree $r\ge1$ is independent of $k$.  The mapped finite element
spaces are
\begin{equation}\label{eq:Wh}
  W_h:=\left\{
    v_h\in C^0(\Gammah):
    v_h\circ F_K^r\in\mathcal P_k(\widehat T)
    \quad\forall K\in\Kh
  \right\}
\end{equation}
and
\begin{equation}\label{eq:Qh}
  Q_h:=\left\{
    q_h\in L^2(\Gammah):
    q_h\circ F_K^r\in\mathcal P_{k-2}(\widehat T)
    \quad\forall K\in\Kh
  \right\}.
\end{equation}
We note that $F_K^r$ is the general computational element map in \eqref{eq:curved-element-map} and not necessarily defined via closest-point projection $p$.  
The choice $\ell=k-2$ is the smallest natural lifting degree: on an affine
element it contains the Laplacian of every degree-$k$ polynomial, while on a curved element it approximates a smooth Laplacian with order $k-1$.  All
constants below may depend on the fixed degrees $k$ and $r$, but not on $h$.
We impose the discrete mean condition through
\begin{equation}\label{eq:Whsharp}
  W_{h,\#}:=\{v_h\in W_h:(v_h,1)_{\Gammah}=0\}.
\end{equation}

For a continuous function that is $H^2$ on each element, recall the two-sided
conormal-flux jump is defined by
\begin{equation}\label{eq:JE}
  J_E(v)
  :=\vnu_E^+\cdot\nablagh v^+
    +\vnu_E^-\cdot\nablagh v^-.
\end{equation}
For a discontinuous scalar $q$, set
\begin{equation*}
  \mean{q}:=\tfrac12(q^++q^-).
\end{equation*}
The definition \eqref{eq:JE} is invariant under interchanging $K^+$ and $K^-$.  
If the adjacent elements are coplanar, it reduces to the usual notion of jumps of the normal derivative.  

\subsection{Lifting and reconstructed Laplacian}

\begin{definition}[Scalar surface lifting]\label{def:lifting}
For $\phi\in L^2(\Eh)$, define $R_h\phi\in Q_h$ by
\begin{equation}\label{eq:lifting}
  (R_h\phi,q_h)_{\Gammah}
  =\sum_{E\in\Eh}(\phi,\mean{q_h})_E
  \qquad\forall q_h\in Q_h.
\end{equation}
\end{definition}
The lifting is global notation for a sum of local edge liftings. The lifting operator has the following stability property.   

\begin{lemma}[Lifting stability]\label{lem:lifting-stability}
There is a constant $C_R>0$ independent of $h$ such that
\begin{equation}\label{eq:lifting-stability}
  \norm{R_h\phi}_{L^2(\Gammah)}^2
  \le C_R\sum_{E\in\Eh}h_E^{-1}\norm{\phi}_{L^2(E)}^2.
\end{equation}
\end{lemma}

\begin{proof}
The result follows by choosing $q_h=R_h\phi$ in \eqref{eq:lifting}, applying the weighted Cauchy-Schwarz, and using the scaled trace estimate \eqref{eq:scaled-trace} with inverse inequality.
\end{proof}

Let $\Pi_{Q_h}:L^2(\Gammah)\to Q_h$ be the elementwise $L^2$ projection.  
Since $R_h\phi\in Q_h$, the exact lifting identity for a general $g\in L^2(\Gammah)$ is
\begin{equation}\label{eq:projection-lifting-identity}
  (g,R_h\phi)_{\Gammah}
  =(\Pi_{Q_h}g,R_h\phi)_{\Gammah}
  =\sum_{E\in\Eh}(\mean{\Pi_{Q_h}g},\phi)_E.
\end{equation}
If $r=1$, the elements are affine and $\deltagh W_h|_K\subset\mathcal P_{k-2}(K)=Q_h|_K$. If $r>1$, $F_K^r$ is not affine and in general
\begin{equation}\label{eq:no-laplacian-inclusion}
  \deltagh W_h\not\subset Q_h.
\end{equation}
Thus the projection in \eqref{eq:projection-lifting-identity} can be removed in the case of $r=1$ but not on a curved mesh.

\begin{definition}[Reconstructed surface Laplacian]
For $v_h\in W_h$, we define the reconstructed surface Laplacian as  
\begin{equation}\label{eq:reconstructed-laplace}
  \tdeltagh v_h:=\deltagh v_h-R_hJ(v_h),
\end{equation}
where $J(v_h)|_E=J_E(v_h)$.
\end{definition}

\subsection{Discrete problem}

Even if $(f,1)_\Gamma=0$, the closest-point extension $f^e$ need not have zero
mean on $\Gammah$.  We therefore define 
\begin{equation}\label{eq:corrected-load}
  f_h:=f^e-\frac{(f^e,1)_{\Gammah}}{|\Gammah|},
  \qquad
  \ell_h(v_h):=(f_h,v_h)_{\Gammah}.
\end{equation}
For a fixed stabilization parameter $\beta>0$, we define the discrete bilinear form 
\begin{equation}\label{eq:discrete-form}
  a_h(v_h,w_h)
  :=(\tdeltagh v_h,\tdeltagh w_h)_{\Gammah}
  +\beta\sum_{E\in\Eh}h_E^{-1}
  (J_E(v_h),J_E(w_h))_E.
\end{equation}
The discrete problem is to find $u_h\in W_{h,\#}$ such that
\begin{equation}\label{eq:discrete-problem}
  a_h(u_h,v_h)=\ell_h(v_h)
  \qquad\forall v_h\in W_{h,\#}.
\end{equation}

Using \eqref{eq:projection-lifting-identity}, the exact expanded form on a
curved mesh is
\begin{align}\label{eq:expanded-form}
  a_h(v_h,w_h)
  ={}&(\deltagh v_h,\deltagh w_h)_{\Gammah}
  -\sum_{E\in\Eh}
    (\mean{\Pi_{Q_h}\deltagh v_h},J_E(w_h))_E
  \notag\\
  &-\sum_{E\in\Eh}
    (\mean{\Pi_{Q_h}\deltagh w_h},J_E(v_h))_E
  +(R_hJ(v_h),R_hJ(w_h))_{\Gammah}
  \notag\\
  &+\beta\sum_{E\in\Eh}h_E^{-1}
    (J_E(v_h),J_E(w_h))_E.
\end{align}
For $r=1$, the projections can be omitted and
\eqref{eq:expanded-form} equals the corresponding symmetric $C^0$ interior
penalty form of \cite{larsson2017continuous} plus the additional positive
lifting-lifting term.

\subsection{Stability and well-posedness}

For a continuous, elementwise-$H^2$ function, we define the DG semi-norm as 
\begin{equation}\label{eq:energy-seminorm}
  \enorm{v}^2
  :=\norm{\deltagh v}_{L^2(\Gammah)}^2
  +\sum_{E\in\Eh}h_E^{-1}\norm{J_E(v)}_{L^2(E)}^2.
\end{equation}

\begin{lemma}[Kernel of $\enorm{\cdot}$]\label{lem:kernel}
If $v_h\in W_h$ and $\enorm{v_h}=0$, then $v_h$ is constant on $\Gammah$.
Consequently, \eqref{eq:energy-seminorm} is a norm on $W_{h,\#}$.
\end{lemma}

\begin{proof}
Elementwise integration by parts on the curved elements, continuity of $v_h$,
and \eqref{eq:JE} give
\begin{equation*}
  \sum_{K\in\Kh}\norm{\nablagh v_h}_{L^2(K)}^2
  =-\sum_{K\in\Kh}(\deltagh v_h,v_h)_K
   +\sum_{E\in\Eh}(J_E(v_h),v_h)_E=0.
\end{equation*}
Thus $v_h$ is constant on every element.  Connectivity and continuity imply
that it is one global constant, which vanishes in $W_{h,\#}$.
\end{proof}

\begin{lemma}[Continuity and coercivity]\label{lem:stability}
For every fixed $\beta>0$, there are constants $C_\beta,c_\beta>0$,
independent of $h$, such that
\begin{align}
  |a_h(v_h,w_h)|
  &\le C_\beta\enorm{v_h}\enorm{w_h},
  \label{eq:continuity}
  \\
  a_h(v_h,v_h)
  &\ge c_\beta\enorm{v_h}^2
  \label{eq:coercivity}
\end{align}
for all $v_h,w_h\in W_h$.  The coercivity constant may tend to zero as
$\beta\downarrow0$.
\end{lemma}

\begin{proof}
Continuity follows from Cauchy-Schwarz and Lemma~\ref{lem:lifting-stability}.  
For coercivity, for $\alpha>1$, applying Young's inequality we derive
\begin{equation*}
  \norm{\tdeltagh v_h}_{L^2(\Gammah)}^2
  \ge(1-\alpha^{-1})
  \norm{\deltagh v_h}_{L^2(\Gammah)}^2-(\alpha-1)C_R
  \sum_{E\in\Eh}h_E^{-1}\norm{J_E(v_h)}_{L^2(E)}^2.
\end{equation*}
Hence
\begin{equation*}
  a_h(v_h,v_h)
  \ge(1-\alpha^{-1})
  \norm{\deltagh v_h}_{L^2(\Gammah)}^2
  +\bigl(\beta-(\alpha-1)C_R\bigr)
  \sum_{E\in\Eh}h_E^{-1}\norm{J_E(v_h)}_{L^2(E)}^2.
\end{equation*}
Then choosing $1<\alpha<1+\beta/C_R$ concludes the proof of \eqref{eq:coercivity}.
\end{proof}

Lemmas~\ref{lem:kernel} and~\ref{lem:stability} and the Lax-Milgram theorem imply that \eqref{eq:discrete-problem} has a unique solution for every fixed $\beta>0$.
We emphasize that the coercivity holds for any fixed $\beta>0$, although the coercivity constant may degenerate as $\beta$ tends to $0$; in practice, there is no need to tune this parameter to above a threshold as in the IPDG method.      

%=======================================================================
\section{Error estimates}\label{sec:error-estimates}
%=======================================================================

\subsection{Exact broken identity and auxiliary estimates}\label{sec:auxiliary-est}

On the exact lifted mesh, define
\begin{equation}\label{def:pw-H2-space}
  H_h^2(\Gamma):=\{v\in C^0(\Gamma):v|_{K^l}\in H^2(K^l)
  \text{ for every }K^l\in\mathcal K_h^l\}.
\end{equation}
For $E^l=\partial(K^+)^l\cap\partial(K^-)^l$, set
\begin{equation*}
  J_{E^l}^{\,l}(w)
  :=\vnu_{E^l}^+\cdot\nablag w^+
   +\vnu_{E^l}^-\cdot\nablag w^-.
\end{equation*}

\begin{lemma}[Exact broken identity]\label{lem:exact-broken}
Let $u\in H^4(\Gamma)$ solve \eqref{eq:strong-problem}.  For every
$w\in H_h^2(\Gamma)$,
\begin{equation}\label{eq:exact-broken}
  (f,w)_\Gamma
  =\sum_{K^l\in\mathcal K_h^l}
    (\deltag u,\deltag w)_{K^l}
  -\sum_{E^l\in\mathcal E_h^l}
    (\deltag u,J_{E^l}^{\,l}(w))_{E^l}.
\end{equation}
\end{lemma}

\begin{proof}
This is obtained directly by integrating $\deltag^2u$ by parts twice on each $K^l$.  The edge terms containing the continuous function $w$ and the conormal derivative of $\deltag u$ cancel, whereas the conormal derivatives of the piecewise-$H^2$ function $w$ enter into $J_{E^l}^{\,l}(w)$.
\end{proof}

Let $I_hu^e\in W_h$ denote the degree-$k$ nodal Lagrange interpolant and define its mean-corrected version by
\begin{equation}\label{eq:mean-corrected-interpolant}
  I_h^\#u^e
  :=I_hu^e-\frac{(I_hu^e,1)_{\Gammah}}{|\Gammah|}
  \in W_{h,\#}.
\end{equation}

We next prove some useful estimates, which extend the lowest order case $r=1$ and $k=2$ described in \cite[Lemmas~3.3, 5.2, 5.3, 5.7, C.1]{larsson2017continuous} to the general situation. 

\begin{lemma}[Interpolation estimates and discrete Poincar\'e inequality]
\label{lem:known-approximation}
For $u\in H^{k+1}(\Gamma)$ and $w_h\in W_h$,
\begin{equation}\label{eq:interpolation-energy}
  \enorm{u^e-I_h^\#u^e}
  \le Ch^{k-1}\norm{u}_{H^{k+1}(\Gamma)},
\end{equation}
and
\begin{equation}\label{eq:lower-order-control}
  \norm{w_h^l}_{L^2(\Gamma)/\R}
  +\norm{\nablag w_h^l}_{L^2(\Gamma)}
  \le C\enorm{w_h}.
\end{equation}
\end{lemma}

\begin{proof}
We define $e_h:=u^e-I_hu^e$. 
Since $I_h^\#u^e-I_hu^e$ is constant, it does not contribute to either $\deltagh e_h$ or $J_E(e_h)$.  
By the standard interpolation theory for regular families of curved parametric elements
\cite{ciarlet1972interpolation,antonietti2014high,demlow2009higher}, together with the Sobolev norm equivalences \eqref{eq:extension-norm-equivalence}-\eqref{eq:broken-full-norm-equivalence}, we have, for $j=1,2$,
\begin{equation}\label{eq:mapped-interpolation-j12}
  \sum_{K\in\Kh}|e_h|_{H^j(K)}^2
  \le
  Ch^{2(k+1-j)}
  \norm{u}_{H^{k+1}(\Gamma)}^2.
\end{equation}
Indeed, after pullback by $F_K^r$, this is the usual
Bramble-Hilbert estimate on the reference triangle, followed by the
uniform scaling estimates for $F_K^r$ and summation over the mesh.

In particular, $\norm{\deltagh e_h}_{L^2(\Gammah)}\le Ch^{k-1}\norm{u}_{H^{k+1}(\Gamma)}$.  
Moreover, by the scaled trace inequality \eqref{eq:scaled-trace} we derive 
\begin{align*}
  \sum_{E\in\Eh}h_E^{-1}\norm{J_E(e_h)}_{L^2(E)}^2
  &\le
  C\sum_{K\in\Kh}h_K^{-1}
    \norm{\nablagh e_h}_{L^2(\partial K)}^2
  \\
  &\le
  C\sum_{K\in\Kh}
  \left(
    h_K^{-2}\norm{\nablagh e_h}_{L^2(K)}^2
    +\norm{D_{\Gammah}^2e_h}_{L^2(K)}^2
  \right)
  \\
  &\le
  Ch^{2(k-1)}
  \norm{u}_{H^{k+1}(\Gamma)}^2.
\end{align*}
This proves \eqref{eq:interpolation-energy}.

To prove \eqref{eq:lower-order-control}, let $c_h:=|\Gammah|^{-1}(w_h,1)_{\Gammah}$. 
Applying the standard Poincar\'e-Friedrichs inequality and the norm equivalences, we have
\begin{equation}\label{eq:uniform-poincare-gammah}
  \norm{w_h-c_h}_{L^2(\Gammah)}
  \le C\norm{\nablagh w_h}_{L^2(\Gammah)}.
\end{equation}
Elementwise integration by parts and the continuity of $w_h$ yield
\begin{equation}\label{eq:grad-by-parts}
  \norm{\nablagh w_h}_{L^2(\Gammah)}^2
  =-(\deltagh w_h,w_h-c_h)_{\Gammah}+\sum_{E\in\Eh}(J_E(w_h),w_h-c_h)_E.
\end{equation}
The scaled trace inequality \eqref{eq:scaled-trace} and
\eqref{eq:uniform-poincare-gammah} imply
\begin{equation*}
  \sum_{E\in\Eh}h_E
  \norm{w_h-c_h}_{L^2(E)}^2
  \le
  C\left(
    \norm{w_h-c_h}_{L^2(\Gammah)}^2
    +h^2\norm{\nablagh w_h}_{L^2(\Gammah)}^2
  \right)
  \le
  C\norm{\nablagh w_h}_{L^2(\Gammah)}^2.
\end{equation*}
Consequently, combining it with \eqref{eq:uniform-poincare-gammah} and \eqref{eq:grad-by-parts} and using the weighted Cauchy-Schwarz inequality, we obtain   
\begin{equation*}
  \norm{\nablagh w_h}_{L^2(\Gammah)}^2
  \le
  C\enorm{w_h}
  \norm{\nablagh w_h}_{L^2(\Gammah)}.
\end{equation*}
Then the estimate \eqref{eq:lower-order-control} follows from \eqref{eq:uniform-poincare-gammah} and the norm equivalences.
\end{proof}

We then show some geometric approximation errors related to operators. 
\begin{lemma}[Geometric operator comparisons]
\label{lem:known-operator-comparisons}
Let $z\in H^2(\Gamma)$ and $w_h\in W_h$.  Then
\begin{align}
\norm{\deltagh z^e-(\deltag z)^e}_{L^2(\Gammah)}&\le Ch^r\norm{z}_{H^2(\Gamma)},
  \label{eq:laplacian-comparison-smooth}
  \\
  \norm{\deltagh w_h-(\deltag w_h^l)^e}_{L^2(\Gammah)}
  &\le Ch^r\enorm{w_h},
  \label{eq:laplacian-comparison-discrete}
  \\
  \left(
    \sum_{E\in\Eh}h_E^{-1}
    \norm{(J_{E^l}^{\,l}(w_h^l))^e-J_E(w_h)}_{L^2(E)}^2
  \right)^{1/2}
  &\le Ch^r\enorm{w_h},
  \label{eq:jump-geometry-comparison}
\end{align}
If, in addition, $z\in H^3(\Gamma)$, then
\begin{equation}
  \left(
    \sum_{E\in\Eh}h_E
    \norm{(\deltag z)^e-
      \mean{\deltagh z^e}}_{L^2(E)}^2
  \right)^{1/2}
  \le Ch^r\norm{z}_{H^3(\Gamma)}.
  \label{eq:edge-laplacian-comparison}
\end{equation}
\end{lemma}

\begin{proof}
Applying Lemma~\ref{lem:local-laplacian-perturbation}, summing over the mesh, and using the broken Sobolev norm equivalences in Lemma \ref{lem:norm-trace} lead to 
\[
  \norm{\deltagh z^e-(\deltag z)^e}_{L^2(\Gammah)}
  \le
  C\left(
    h^r\norm{z}_{H^1(\Gamma)}
    +h^{r+1}\norm{z}_{H^2(\Gamma)}
  \right),
\]
which proves \eqref{eq:laplacian-comparison-smooth}.
Taking $v=w_h$ in Lemma \ref{lem:local-laplacian-perturbation} and using an inverse estimate
\[
  \norm{D_{\Gammah}^2w_h}_{L^2(K)}
  \le Ch_K^{-1}\norm{\nablagh w_h}_{L^2(K)}
\]
gives 
\[
  \norm{\deltagh w_h-(\deltag w_h^l)^e}_{L^2(\Gammah)}
  \le
  Ch^r\norm{\nablagh w_h}_{L^2(\Gammah)}.
\]
Estimate \eqref{eq:laplacian-comparison-discrete} now follows from
\eqref{eq:lower-order-control} and the norm equivalences.

For the jump comparison, recalling the tangential gradient relation $(\nablag w_h^l)^e=B_h^{-T}\nablagh w_h$, we have
\begin{equation}\label{eq:jump-transformation}
  (J_{E^l}^{\,l}(w_h^l))^e-J_E(w_h)
  =
  \sum_{\sigma\in\{+,-\}}
  \left(
    B_h^{-1}(\vnu_{E^l}^{\,\sigma})^e-\vnu_E^\sigma
  \right)
  \cdot\nablagh w_h^\sigma.
\end{equation}
Using the definition \eqref{eq:B-map} of $B_h$ and the fact that $\vnu_E^\sigma$ is tangential to $\Gamma_h$, we derive 
\[
  B_h\vnu_E^\sigma
  =
  P^e(I-d\mathcal W)\vnu_E^\sigma.
\]
By the geometric estimates \eqref{eq:individual-conormal-comparison} and \eqref{eq:dist-vel}, we have  
\[
  \norm{(\vnu_{E^l}^{\,\sigma})^e
        -P^e\vnu_E^\sigma}_{L^\infty(E)}
  \le Ch^{r+1},
  \qquad
  \norm{d}_{L^\infty(E)}\le Ch^{r+1},
\]
and hence the uniform boundedness of $B_h^{-1}$ \eqref{eq:Bh-bounded} implies
\begin{equation}\label{eq:inverse-conormal-comparison}
  \norm{
    B_h^{-1}(\vnu_{E^l}^{\,\sigma})^e-\vnu_E^\sigma
  }_{L^\infty(E)}
  \le C\norm{
    (\vnu_{E^l}^{\,\sigma})^e- B_h\vnu_E^\sigma
  }_{L^\infty(E)}
  \le Ch^{r+1}.
\end{equation}
Combining \eqref{eq:jump-transformation}, \eqref{eq:inverse-conormal-comparison}, and the scaled trace inequality \eqref{eq:scaled-trace} along with the inverse estimates give 
\begin{align*}
  \sum_{E\in\Eh}h_E^{-1}
  \norm{(J_{E^l}^{\,l}(w_h^l))^e-J_E(w_h)}_{L^2(E)}^2
  &\le
  C\sum_{K\in\Kh}
    h_K^{2r+1}
    \norm{\nablagh w_h}_{L^2(\partial K)}^2
  \\
  &\le
  Ch^{2r}\norm{\nablagh w_h}_{L^2(\Gammah)}^2
  \le
  Ch^{2r}\enorm{w_h}^2.
\end{align*}
This proves \eqref{eq:jump-geometry-comparison}.

Moreover, applying Lemma~\ref{lem:local-laplacian-perturbation}, we derive
\[
  |\deltagh z^e-(\deltag z)^e|
  \le
  C\left(
    h_K^r|(\nablag z)^e|
    +h_K^{r+1}|(D_\Gamma^2z)^e|
  \right).
\]
Since $(\deltag z)^e-\mean{\deltagh z^e}=-\mean{\deltagh z^e-(\deltag z)^e}$, the scaled trace inequality \eqref{eq:scaled-trace} yields
\begin{align*}
  \sum_{E\in\Eh}h_E
  \norm{(\deltag z)^e-\mean{\deltagh z^e}}_{L^2(E)}^2
  &\le
  C\sum_{K\in\Kh}
  \left(
    h_K^{2r+1}
      \norm{(\nablag z)^e}_{L^2(\partial K)}^2
    +
    h_K^{2r+3}
      \norm{(D_\Gamma^2z)^e}_{L^2(\partial K)}^2
  \right)
  \\
  &\le
  Ch^{2r}\norm{z}_{H^3(\Gamma)}^2,
\end{align*}
which concludes the proof of \eqref{eq:edge-laplacian-comparison}.
\end{proof}

\begin{lemma}[Geometric load consistency]
\label{lem:known-load-consistency}
For $f\in L^2(\Gamma)$ with $(f,1)_\Gamma=0$ and $w_h\in W_h$,
\begin{equation}\label{eq:load-geometric-consistency}
  \left|(f,w_h^l)_\Gamma-\ell_h(w_h)\right|
  \le Ch^{r+1}\norm{f}_{L^2(\Gamma)}\enorm{w_h}.
\end{equation}
\end{lemma}

\begin{proof}
Let $c_h:=|\Gammah|^{-1}(w_h,1)_{\Gammah}$. 
Since both $f$ and $f_h$ have zero mean,
\[
  (f,w_h^l)_\Gamma=(f,w_h^l-c_h)_\Gamma,
  \qquad
  \ell_h(w_h)=\ell_h(w_h-c_h).
\]
Moreover, the definition \eqref{eq:corrected-load} of $f_h$ and $(w_h-c_h,1)_{\Gammah}=0$ imply
\[
  \ell_h(w_h-c_h)
  =(f^e,w_h-c_h)_{\Gammah}.
\]
Changing variables from $\Gamma$ to $\Gammah$ therefore gives
\[
  (f,w_h^l)_\Gamma-\ell_h(w_h)
  =
  ((\mu_h-1)f^e,w_h-c_h)_{\Gammah}.
\]
Consequently, by \eqref{eq:mu-errors}, the norm equivalences, and \eqref{eq:lower-order-control},
\begin{align*}
  \left|(f,w_h^l)_\Gamma-\ell_h(w_h)\right|
  &\le
  \norm{\mu_h-1}_{L^\infty(\Gammah)}
  \norm{f^e}_{L^2(\Gammah)}
  \norm{w_h-c_h}_{L^2(\Gammah)}
  \\
  &\le
  Ch^{r+1}\norm{f}_{L^2(\Gamma)}\enorm{w_h}.
\end{align*}
\end{proof}

\subsection{Projection and consistency}\label{sec:proj-consistency}
We first prove an estimate related to the projection onto $Q_h$. 

\begin{lemma}[Projection residual]
\label{lem:projection-residual}
For $z\in H^{k+1}(\Gamma)$ and $w_h\in W_h$,
\begin{equation}\label{eq:projection-consistency}
  \left|
    \sum_{E\in\Eh}
    (\mean{(I-\Pi_{Q_h})(\deltag z)^e},J_E(w_h))_E
  \right|
  \le Ch^{k-1}\norm{z}_{H^{k+1}(\Gamma)}\enorm{w_h}.
\end{equation}
\end{lemma}

\begin{proof}
By the scaled trace inequality \eqref{eq:scaled-trace}, the approximation error of the local $L^2$ projection $\Pi_{Q_h}$ on each element $K\in\Kh$, and the norm equivalences in Lemma~\ref{lem:norm-trace}, one obtains
\begin{equation}\label{eq:projection-edge-estimate}
\left(
    \sum_{E\in\Eh}h_E
    \norm{\mean{(I-\Pi_{Q_h})(\deltag z)^e}}_{L^2(E)}^2
  \right)^{1/2}
  \le Ch^{k-1}\norm{z}_{H^{k+1}(\Gamma)}.
\end{equation}
Combining \eqref{eq:projection-edge-estimate} with the weighted Cauchy-Schwarz inequality and the definition \eqref{eq:energy-seminorm} proves \eqref{eq:projection-consistency}.
\end{proof}

We extend $a_h$ to $H_h^2(\Gamma_h)\times W_h$, with
\begin{equation}\label{def:pw-H2-space-h}
  H_h^2(\Gamma_h):=\{v\in C^0(\Gamma_h):v|_{K}\in H^2(K)
  \text{ for every }K\in\mathcal K_h\},
\end{equation}
such that for $(z,w_h)\in H_h^2(\Gamma_h)\times W_h$,  
\begin{equation}\label{eq:extended-reconstructed-form}
\begin{aligned}
  a_h(z,w_h)
  &:=(\deltagh z-R_hJ(z),
      \deltagh w_h-R_hJ(w_h))_{\Gammah}+
  \beta\sum_{E\in\Eh}h_E^{-1}
  (J_E(z),J_E(w_h))_E.
\end{aligned}
\end{equation}
If $z,w_h\in W_h$, this agrees with \eqref{eq:discrete-form}.
The extended form also satisfies $|a_h(z,w_h)|\le C_\beta\enorm{z}\enorm{w_h}$,
and the proof follows directly from Lemma \ref{lem:stability}. 

\begin{lemma}[Approximation by the reconstructed Laplacian]
\label{lem:reconstructed-laplacian-approximation}
For $z\in H^2(\Gamma)$,
\begin{equation}\label{eq:reconstructed-exact-extension}
  \norm{\tdeltagh z^e-(\deltag z)^e}_{L^2(\Gamma_h)}
  \le
  Ch^r\norm{z}_{H^2(\Gamma)}.
\end{equation}
If $z\in H^{k+1}(\Gamma)$, then, with $I_h^\#z^e$ defined as in
\eqref{eq:mean-corrected-interpolant},
\begin{equation}\label{eq:reconstructed-interpolant}
  \norm{\tdeltagh I_h^\#z^e-(\deltag z)^e}_{L^2(\Gamma_h)}
  \le
  C\left(
    h^{k-1}\norm{z}_{H^{k+1}(\Gamma)}
    +h^r\norm{z}_{H^2(\Gamma)}
  \right).
\end{equation}
\end{lemma}

\begin{proof}
The smooth Laplacian comparison
\eqref{eq:laplacian-comparison-smooth}, the lifting stability
\eqref{eq:lifting-stability}, and the geometric jump estimate
\eqref{eq:smooth-geometric-jump-global} give
\begin{align*}
  \norm{\tdeltagh z^e-(\deltag z)^e}_{L^2(\Gamma_h)}
  &\le
  \norm{\deltagh z^e-(\deltag z)^e}_{L^2(\Gamma_h)}
  +\norm{R_hJ(z^e)}_{L^2(\Gamma_h)}\le
  Ch^r\norm{z}_{H^2(\Gamma)}.
\end{align*}
This proves \eqref{eq:reconstructed-exact-extension}.  Moreover, since the lifting stability implies that $ \norm{\tdeltagh v}_{L^2(\Gamma_h)}\le C\enorm{v}$ for any $v\in H_h^2(\Gamma_h)$,  the triangle inequality,
\eqref{eq:interpolation-energy}, and
\eqref{eq:reconstructed-exact-extension} yield
\begin{align*}
  \norm{\tdeltagh I_h^\#z^e-(\deltag z)^e}_{L^2(\Gamma_h)}
  &\le
  \norm{\tdeltagh(I_h^\#z^e-z^e)}_{L^2(\Gamma_h)}
  +\norm{\tdeltagh z^e-(\deltag z)^e}_{L^2(\Gamma_h)}
  \\
  &\le
  C\left(
    h^{k-1}\norm{z}_{H^{k+1}(\Gamma)}
    +h^r\norm{z}_{H^2(\Gamma)}
  \right),
\end{align*}
which proves \eqref{eq:reconstructed-interpolant}.
\end{proof}

Towards estimating the consistency error, we next investigate the geometric consistency error between the following terms appearing in the bilinear form $a$ or $a_h$. 

\begin{lemma}[Comparison of the principal terms]
\label{lem:principal-comparison}
Let $u\in H^2(\Gamma)$ and $w_h\in W_h$.  Then
\begin{align}
  \left|
    (\deltagh u^e,\deltagh w_h)_{\Gammah}
    -\sum_{K^l\in\mathcal K_h^l}
      (\deltag u,\deltag w_h^l)_{K^l}
  \right|
  &\le Ch^r\norm{u}_{H^2(\Gamma)}\enorm{w_h},
  \label{eq:direct-element-comparison}
  \\
  \intertext{If, in addition, $u\in H^3(\Gamma)$, then}
  \left|
    \sum_{E^l\in\mathcal E_h^l}
      (\deltag u,J_{E^l}^{\,l}(w_h^l))_{E^l}
    -\sum_{E\in\Eh}
      (\mean{\deltagh u^e},J_E(w_h))_E
  \right|
  &\le Ch^r\norm{u}_{H^3(\Gamma)}\enorm{w_h}.
  \label{eq:direct-edge-comparison}
\end{align}
Finally, if $u\in H^{k+1}(\Gamma)$, then
\begin{equation}\label{eq:reconstructed-principal-comparison}
\begin{aligned}
  \bigg|
    ((\deltag u)^e,\tdeltagh w_h)_{\Gamma_h}
    &-
    \bigg[
      \sum_{K^l\in\mathcal K_h^l}
        (\deltag u,\deltag w_h^l)_{K^l}
      -
      \sum_{E^l\in\mathcal E_h^l}
        (\deltag u,J_{E^l}^{\,l}(w_h^l))_{E^l}
    \bigg]
  \bigg|
  \\
  &\le
  C\left(
    h^{k-1}\norm{u}_{H^{k+1}(\Gamma)}
    +h^r\norm{u}_{H^3(\Gamma)}
  \right)\enorm{w_h}.
\end{aligned}
\end{equation}
\end{lemma}

\begin{proof}
Changing variables from $K^l$ to $K$ yields 
\begin{equation*}
  \sum_{K^l\in\Kh^l}(\deltag u,\deltag w_h^l)_{K^l}=(\mu_h(\deltag u)^e,(\deltag w_h^l)^e)_{\Gammah}.
\end{equation*}
Then, 
\begin{align}\label{eq: delta-diff-add-sub}
 (\deltagh u^e,\deltagh w_h)_{\Gammah}&-\sum_{K^l\in\mathcal K_h^l}(\deltag u,\deltag w_h^l)_{K^l} = (\deltagh u^e-(\deltag u)^e,\deltagh w_h)_{\Gammah}\\ \nonumber
 &+((1-\mu_h)(\deltag u)^e,\deltagh w_h)_{\Gammah} +(\mu_h(\deltag u)^e,\deltagh w_h-(\deltag w_h^l)^e)_{\Gammah}.
\end{align}

For the first term, we estimate it using \eqref{eq:laplacian-comparison-smooth}. 
For the second term, we invoke the geometric error \eqref{eq:mu-errors}. The third term is handled by the estimate \eqref{eq:laplacian-comparison-discrete}.  
Therefore, \eqref{eq:direct-element-comparison} follows by applying the Cauchy-Schwarz inequality to \eqref{eq: delta-diff-add-sub}.   

For the edge term, we proceed similarly. After changing variables from $E^l$ to $E$,
\begin{equation*}
  \sum_{E^l\in \Eh^l}
  (\deltag u,J_{E^l}^{\,l}(w_h^l))_{E^l}
  =
  \sum_{E\in\Eh}
  (\mu_{h,E}(\deltag u)^e,(J_{E^l}^{\,l}(w_h^l))^e)_E.
\end{equation*}
Therefore,
\begin{align*}
\sum_{E^l\in\mathcal E_h^l}
      &(\deltag u,J_{E^l}^{\,l}(w_h^l))_{E^l}
    -\sum_{E\in\Eh}
      (\mean{\deltagh u^e},J_E(w_h))_E=\sum_{E\in\Eh}((\mu_{h,E}-1)(\deltag u)^e,(J_{E^l}^{\,l}(w_h^l))^e)_E \\
      &+\sum_{E\in\Eh}((\deltag u)^e,(J_{E^l}^{\,l}(w_h^l))^e-J_E(w_h))_E+\sum_{E\in\Eh}((\deltag u)^e-\mean{\deltagh u^e},J_E(w_h))_E.
\end{align*}
The first term is estimated using \eqref{eq:mu-errors}, and we apply \eqref{eq:jump-geometry-comparison} and \eqref{eq:edge-laplacian-comparison} to the second and third terms respectively. Then invoking the weighted Cauchy-Schwarz inequality and scaled trace inequality \eqref{eq:scaled-trace} we conclude the proof of \eqref{eq:direct-edge-comparison}.

It remains to prove \eqref{eq:reconstructed-principal-comparison}.  
By the lifting identity \eqref{eq:projection-lifting-identity},
\[
  ((\deltag u)^e,R_hJ(w_h))_{\Gamma_h}
  =
  \sum_{E\in\Eh}
  (\mean{\Pi_{Q_h}(\deltag u)^e},J_E(w_h))_E.
\]
Adding and subtracting the two principal terms already compared above gives
\begin{align*}
  &((\deltag u)^e,\tdeltagh w_h)_{\Gamma_h}
  -
  \bigg[
    \sum_{K^l\in\mathcal K_h^l}
      (\deltag u,\deltag w_h^l)_{K^l}
    -
    \sum_{E^l\in\mathcal E_h^l}
      (\deltag u,J_{E^l}^{\,l}(w_h^l))_{E^l}
  \bigg]
  \\
  ={}&
  \bigg[
    (\deltagh u^e,\deltagh w_h)_{\Gamma_h}
    -
    \sum_{K^l\in\mathcal K_h^l}
      (\deltag u,\deltag w_h^l)_{K^l}
  \bigg]
  -
  (\deltagh u^e-(\deltag u)^e,\deltagh w_h)_{\Gamma_h}
  \\
  &+
  \bigg[
    \sum_{E^l\in\mathcal E_h^l}
      (\deltag u,J_{E^l}^{\,l}(w_h^l))_{E^l}
    -
    \sum_{E\in\Eh}
      (\mean{\deltagh u^e},J_E(w_h))_E
  \bigg]
  \\
  &+
  \sum_{E\in\Eh}
  \left(
    \mean{\deltagh u^e-(\deltag u)^e}
    +
    \mean{(I-\Pi_{Q_h})(\deltag u)^e},
    J_E(w_h)
  \right)_E.
\end{align*}
The first three terms are bounded by \eqref{eq:direct-element-comparison}, \eqref{eq:laplacian-comparison-smooth}, and \eqref{eq:direct-edge-comparison}, respectively.  
The two terms in the last line are bounded by \eqref{eq:edge-laplacian-comparison} along with the weighted Cauchy-Schwarz inequality, and Lemma~\ref{lem:projection-residual}.  
This proves \eqref{eq:reconstructed-principal-comparison}.
\end{proof}

\begin{lemma}[Consistency of the exact extension]
\label{lem:exact-extension-consistency}
Let $u\in H^{s_k}_{\#}(\Gamma)$ solve \eqref{eq:strong-problem}.  Then, for every
$w_h\in W_h$,
\begin{equation}\label{eq:exact-extension-discrete-load}
|a_h(u^e,w_h)-\ell_h(w_h)|
  \le
  C\left(
    h^{k-1}\norm{u}_{H^{k+1}(\Gamma)}
    +h^r\norm{u}_{H^3(\Gamma)}
    +h^{r+1}\norm{u}_{H^4(\Gamma)}
  \right)\enorm{w_h}.
\end{equation}
\end{lemma}

\begin{proof}
We compute 
\begin{align*}
  a_h(u^e,w_h)-\ell_h(w_h)
  ={}&
  (\tdeltagh u^e-(\deltag u)^e,\tdeltagh w_h)_{\Gamma_h}
  +
  \bigl[
    ((\deltag u)^e,\tdeltagh w_h)_{\Gamma_h}
    -(f,w_h^l)_\Gamma
  \bigr]
  \\
  &+
  \beta\sum_{E\in\Eh}h_E^{-1}
  (J_E(u^e),J_E(w_h))_E
  +
  \bigl[
    (f,w_h^l)_\Gamma-\ell_h(w_h)
  \bigr].
\end{align*}
The first term is bounded by \eqref{eq:reconstructed-exact-extension} and $\norm{\tdeltagh w_h}_{L^2(\Gamma_h)}\le C\enorm{w_h}$.  
Substituting the exact broken identity \eqref{eq:exact-broken} into the second term, we then apply \eqref{eq:reconstructed-principal-comparison} to get the desired estimates.  
The stabilization term is bounded by \eqref{eq:smooth-geometric-jump-global} and the weighted Cauchy-Schwarz inequality. 
The final term is bounded by Lemma~\ref{lem:known-load-consistency}.  Consequently,
\begin{align*}
  |a_h(u^e,w_h)-\ell_h(w_h)|
  \le C\Big(
    &h^{k-1}\norm{u}_{H^{k+1}(\Gamma)}
    +h^r\norm{u}_{H^3(\Gamma)}
    +h^{r+1}\norm{f}_{L^2(\Gamma)}
  \Big)\enorm{w_h}.
\end{align*}
Since $f=\deltag^2u$,
$\norm{f}_{L^2(\Gamma)}\le C\norm{u}_{H^4(\Gamma)}$, and
\eqref{eq:exact-extension-discrete-load} follows.
\end{proof}

\begin{theorem}[Consistency error]
\label{thm:consistency-error}
Let $u\in H^{s_k}_{\#}(\Gamma)$ solve \eqref{eq:strong-problem}.  Then, for every
$w_h\in W_h$,
\begin{equation}\label{eq:proved-consistency-estimate}
  |a_h(I_h^\#u^e,w_h)-\ell_h(w_h)|
  \le
  C\left(
    h^{k-1}\norm{u}_{H^{k+1}(\Gamma)}
    +h^r\norm{u}_{H^3(\Gamma)}
    +h^{r+1}\norm{u}_{H^4(\Gamma)}
  \right)\enorm{w_h}.
\end{equation}
\end{theorem}

\begin{proof}
Adding and subtracting $a_h(u^e,w_h)$ gives
\begin{align*}
  a_h(I_h^\#u^e,w_h)-\ell_h(w_h)
  ={}&a_h(I_h^\#u^e-u^e,w_h)
  +a_h(u^e,w_h)-\ell_h(w_h).
\end{align*}
The first term is bounded using the continuity of $a_h$ and \eqref{eq:interpolation-energy}, and the second is bounded by Lemma~\ref{lem:exact-extension-consistency}.
\end{proof}

\begin{remark}[Regularity bookkeeping and the mean-zero condition]
\label{rem:regularity-mean-zero}
The separate Sobolev indices in
\eqref{eq:proved-consistency-estimate} identify where regularity is used.
The reconstructed Laplacian of the exact extension requires only
$u\in H^2(\Gamma)$, the geometric edge comparison requires
$u\in H^3(\Gamma)$, and interpolation and the degree-$(k-2)$ projection
residual require $u\in H^{k+1}(\Gamma)$.  The assumption
$u\in H^4(\Gamma)$ is used for the strong equation
$f=\deltag^2u\in L^2(\Gamma)$, the exact broken identity, and the load term;
the latter carries the higher factor $h^{r+1}$.  Hence
$s_k=\max\{k+1,4\}$ is the smallest common integer regularity index for the
argument, although it is not needed in every intermediate estimate.

No mean-zero condition on $w_h$ is used in the consistency analysis through
Theorem~\ref{thm:consistency-error}.  Indeed,
$a_h(v,1)=0$ and $\ell_h(1)=0$, while $\enorm{\cdot}$ is insensitive to
constants.  The restriction to $W_{h,\#}$ enters in the Strang estimate and
the discrete problem because it removes the constant kernel, makes
\eqref{eq:energy-seminorm} a norm, and permits coercivity to control the
discrete error uniquely.
\end{remark}

\subsection{Strang estimate and convergence rate in energy norm}

\begin{lemma}[Strang estimate]
\label{lem:strang}
Let $u$ solve \eqref{eq:strong-problem} and $u_h$ solve
\eqref{eq:discrete-problem}.  Then
\begin{equation}\label{eq:strang}
  \enorm{u^e-u_h}
  \le C_\beta
  \inf_{v_h\in W_{h,\#}}
  \left(
    \enorm{u^e-v_h}
    +\sup_{0\ne w_h\in W_{h,\#}}
    \frac{|a_h(v_h,w_h)-\ell_h(w_h)|}{\enorm{w_h}}
  \right).
\end{equation}
\end{lemma}

\begin{proof}
For arbitrary $v_h\in W_{h,\#}$, the triangle inequality and coercivity give
\begin{align*}
  \enorm{u^e-u_h}
  &\le\enorm{u^e-v_h}+\enorm{v_h-u_h},
  \\
  c_\beta\enorm{v_h-u_h}
  &\le
  \sup_{0\ne w_h\in W_{h,\#}}
  \frac{|a_h(v_h,w_h)-\ell_h(w_h)|}{\enorm{w_h}}.
\end{align*}
Taking the infimum over $v_h$ proves \eqref{eq:strang}.
\end{proof}

\begin{theorem}[Energy error estimate]\label{thm:energy}
Let $u\in H^{s_k}_{\#}(\Gamma)$ solve \eqref{eq:strong-problem} and let $u_h$ solve
\eqref{eq:discrete-problem}. Then, 
\begin{equation}\label{eq:energy-error}
  \enorm{u^e-u_h}\le C(h^{k-1}+h^r)\norm{f}_{H^{s_k-4}(\Gamma)}.
\end{equation}
Consequently, the energy convergence rate is
$\min\{k-1,r\}$.  In particular, the optimal degree-$k$ rate $k-1$ is
obtained whenever $r\ge k-1$.
\end{theorem}

\begin{proof}
We choose $v_h=I_h^\#u^e$ in Lemma~\ref{lem:strang} and apply \eqref{eq:interpolation-energy}, Theorem~\ref{thm:consistency-error}, and the regularity estimate \eqref{eq:higher-regularity} to conclude 
\begin{equation}\label{eq:energy-error-2}
\begin{aligned}
  \enorm{u^e-u_h}
  &\le
  C\left(
    h^{k-1}\norm{u}_{H^{k+1}(\Gamma)}
    +h^r\norm{u}_{H^3(\Gamma)}
    +h^{r+1}\norm{u}_{H^4(\Gamma)}
  \right)
  \\
  &\le C(h^{k-1}+h^r)\norm{u}_{H^{s_k}(\Gamma)}
  \le C(h^{k-1}+h^r)\norm{f}_{H^{s_k-4}(\Gamma)}.
\end{aligned}
\end{equation}
\end{proof}

\begin{corollary}[Convergence of the discrete reconstructed Laplacian]
\label{cor:reconstructed-laplacian}
Let $u\in H^{s_k}_{\#}(\Gamma)$ solve \eqref{eq:strong-problem} and let $u_h$ solve \eqref{eq:discrete-problem}. Then, 
\begin{equation}\label{eq:reconstructed-solution-error}
\begin{aligned}
  \norm{\tdeltagh u_h-(\deltag u)^e}_{L^2(\Gamma_h)}
  \le C(h^{k-1}+h^r)\norm{f}_{H^{s_k-4}(\Gamma)}.
\end{aligned}
\end{equation}
\end{corollary}

\begin{proof}
By the triangle inequality and \eqref{eq:reconstructed-interpolant},
\begin{align*}
  \norm{\tdeltagh u_h-(\deltag u)^e}_{L^2(\Gamma_h)}
  &\le
  \norm{\tdeltagh(u_h-I_h^\#u^e)}_{L^2(\Gamma_h)}
  +
  \norm{\tdeltagh I_h^\#u^e-(\deltag u)^e}_{L^2(\Gamma_h)}
  \\
  &\le
  C\enorm{u_h-I_h^\#u^e}
  +
  C\left(
    h^{k-1}\norm{u}_{H^{k+1}(\Gamma)}
    +h^r\norm{u}_{H^2(\Gamma)}
  \right).
\end{align*}
The triangle inequality in $\enorm{\cdot}$, the interpolation estimate
\eqref{eq:interpolation-energy}, and Theorem~\ref{thm:energy} prove
\eqref{eq:reconstructed-solution-error}.
\end{proof}

The three terms in the first bound of \eqref{eq:energy-error-2} have distinct origins.  
The term $h^{k-1}$ is both the best-approximation error of the $\mathcal P_k$ solution space in the energy norm and the projection residual generated by the degree-$(k-2)$ lifting space.  The term $h^r$ is the leading geometric consistency error.  
The $h^{r+1}$ term is the higher-order geometric
load error. 
Raising the lifting degree above $k-2$ can reduce the projection residual for sufficiently smooth solutions, but it cannot improve the $h^{k-1}$ best-approximation rate.  

\subsection{An $L^2$-error estimate}
\label{subsec:L2-error}
We derive an $L^2$-error estimate by adapting the duality argument of \cite[Theorem~5.11]{larsson2017continuous} to arbitrary $r\ge1$ and $k\ge2$ and to the reconstructed surface Laplacian formulation.  
Set
\begin{equation}\label{eq:dual-orders}
  \alpha_k:=\min\{k-1,2\},
  \qquad
  q_k:=k-1+\alpha_k
  =
  \begin{cases}
    2,   & k=2,\\
    k+1, & k\ge3.
  \end{cases}
\end{equation}
The number $\alpha_k$ is the energy-norm approximation order of the dual
solution under $H^4(\Gamma)$ regularity.
For a closed surface $\Gamma$, define
\begin{equation}\label{eq:mean-projection-L2}
   \norm{v}_{L^2(\Gamma)/\mathbb R}
  :=
  \norm{v-|\Gamma|^{-1}(v,1)_\Gamma}_{L^2(\Gamma)}.
\end{equation}

For the duality argument, we use the same expression
\eqref{eq:extended-reconstructed-form} to extend $a_h$ to $H_h^2(\Gamma_h)\times H_h^2(\Gamma_h)$.
The lifting stability estimate implies the continuity $ |a_h(v,w)|\le C_\beta\enorm{v}\enorm{w}$. 

\begin{lemma}[Consistency tested by a dual interpolation error]
\label{lem:dual-interpolation-consistency}
Let $u\in H^{s_k}_\#(\Gamma)$ solve
\eqref{eq:strong-problem}, let $\phi\in H^4(\Gamma)$, and set $\phi_h:=I_h^\#\phi^e$ and $\eta_h:=\phi_h-\phi^e$.
Then
\begin{equation}\label{eq:dual-interpolation-consistency}
  \left|
    a_h(u^e,\eta_h)
    -(f,\eta_h^l)_\Gamma
  \right|
  \le
  C\left(
    h^{k-1+\alpha_k}
    +
    h^{r+\alpha_k}
  \right)
  \norm{u}_{H^{s_k}(\Gamma)}
  \norm{\phi}_{H^4(\Gamma)}.
\end{equation}
\end{lemma}

\begin{proof}
Note that the constant used in the definition of $I_h^\#$ does not contribute to either side of \eqref{eq:dual-interpolation-consistency}, since $a_h(v,1)=0$ and $(f,1)_\Gamma=0$.  
Standard interpolation therefore yields 
\begin{align}
  &\enorm{\eta_h}
  \le
  Ch^{\alpha_k}\norm{\phi}_{H^4(\Gamma)},
  \label{eq:dual-eta-energy}
  \\
  &\norm{\nablagh\eta_h}_{L^2(\Gamma_h)}
  +
  \left(
    \sum_{K\in\Kh}
    h_K^2
    \norm{D_{\Gamma_h}^2\eta_h}_{L^2(K)}^2
  \right)^{1/2}
  \le
  Ch^{\alpha_k+1}\norm{\phi}_{H^4(\Gamma)}.
  \label{eq:dual-eta-scaled}
\end{align}
In particular, $ \norm{\tdeltagh\eta_h}_{L^2(\Gamma_h)}\le C\enorm{\eta_h}$.
Using the exact broken identity \eqref{eq:exact-broken} and the lifting identity \eqref{eq:projection-lifting-identity}, we expand   
\begin{align*}
  a_h(u^e,\eta_h)
  -(f,\eta_h^l)_\Gamma
  ={}&
  \bigl(
    \tdeltagh u^e-(\deltag u)^e,
    \tdeltagh\eta_h
  \bigr)_{\Gamma_h}
  +
  \beta\sum_{E\in\Eh}h_E^{-1}
  (J_E(v_h),J_E(w_h))_E
  \\
  &+
  \left[
    ((\deltag u)^e,\deltagh\eta_h)_{\Gamma_h}
    -
    \sum_{K^l\in\mathcal K_h^l}
    (\deltag u,\deltag\eta_h^l)_{K^l}
  \right]
  \\
  &+
  \left[
    \sum_{E^l\in\mathcal E_h^l}
    (\deltag u,J_{E^l}^{\,l}(\eta_h^l))_{E^l}
    -
    \sum_{E\in\Eh}
    \bigl(
      \mean{\Pi_{Q_h}(\deltag u)^e},
      J_E(\eta_h)
    \bigr)_E
  \right].
\end{align*}

The reconstructed-Laplacian comparison
\eqref{eq:reconstructed-exact-extension}, the smooth-jump estimate
\eqref{eq:smooth-geometric-jump-global}, and
\eqref{eq:dual-eta-energy} imply
\begin{align*}
  \left|
    \bigl(
      \tdeltagh u^e-(\deltag u)^e,
      \tdeltagh\eta_h
    \bigr)_{\Gamma_h}
  \right|
  +
  \left|
   \beta\sum_{E\in\Eh}h_E^{-1}
  (J_E(u^e),J_E(\eta_h))_E
  \right|
  \le
  Ch^{r+\alpha_k}
  \norm{u}_{H^2(\Gamma)}
  \norm{\phi}_{H^4(\Gamma)}.
\end{align*}

Then, changing variables from $K^l$ to $K$ and using the local operator comparison \eqref{eq:local-laplacian-perturbation} gives
\begin{align*}
  &\left|
    ((\deltag u)^e,\deltagh\eta_h)_{\Gamma_h}
    -
    \sum_{K^l\in\mathcal K_h^l}
    (\deltag u,\deltag\eta_h^l)_{K^l}
  \right|
  \\
  &\qquad\le
  C\norm{u}_{H^2(\Gamma)}
  \left(
    h^{r+1}\enorm{\eta_h}
    +
    h^r\norm{\nablagh\eta_h}_{L^2(\Gamma_h)}
    +
    h^{r+1}
    \left(
      \sum_{K\in\Kh}
      \norm{D_{\Gamma_h}^2\eta_h}_{L^2(K)}^2
    \right)^{1/2}
  \right)
  \\
  &\qquad\le
  Ch^{r+\alpha_k+1}
  \norm{u}_{H^2(\Gamma)}
  \norm{\phi}_{H^4(\Gamma)}.
\end{align*}

Moreover, after changing variables on the lifted edges, we have
\begin{align*} 
&\left[
    \sum_{E^l\in\mathcal E_h^l}
    (\deltag u,J_{E^l}^{\,l}(\eta_h^l))_{E^l}
    -
    \sum_{E\in\Eh}
    \bigl(
      \mean{\Pi_{Q_h}(\deltag u)^e},
      J_E(\eta_h)
    \bigr)_E
  \right] \\
  &=\sum_{E\in\Eh}
  \bigl(
    (\mu_{h,E}-1)(\deltag u)^e,
    (J_{E^l}^{\,l}(\eta_h^l))^e
  \bigr)_E
  +
  \sum_{E\in\Eh}
  \bigl(
    (\deltag u)^e,
    (J_{E^l}^{\,l}(\eta_h^l))^e-J_E(\eta_h)
  \bigr)_E
  \\
  &\quad+
  \sum_{E\in\Eh}
  \bigl(
    \mean{(I-\Pi_{Q_h})(\deltag u)^e},
    J_E(\eta_h)
  \bigr)_E.
\end{align*}
The first two sums are bounded by the edge-measure estimate \eqref{eq:mu-errors}, the conormal comparison arguments (which follow directly from the proof of \eqref{eq:jump-geometry-comparison}, but cannot be applied directly since $\eta_h\notin W_h$), the scaled trace inequality \eqref{eq:scaled-trace}, and \eqref{eq:dual-eta-energy}--\eqref{eq:dual-eta-scaled}.
They are of order
\[
  Ch^{r+\alpha_k+1}
  \norm{u}_{H^3(\Gamma)}
  \norm{\phi}_{H^4(\Gamma)}.
\]
For the projection term, \eqref{eq:projection-edge-estimate} and the weighted Cauchy-Schwarz give
\begin{align*}
  &\left|
    \sum_{E\in\Eh}
    \bigl(
      \mean{(I-\Pi_{Q_h})(\deltag u)^e},
      J_E(\eta_h)
    \bigr)_E
  \right|
  \le
  Ch^{k-1}
  \norm{u}_{H^{k+1}(\Gamma)}
  \enorm{\eta_h}
  \\
  &\qquad\le
  Ch^{k-1+\alpha_k}
  \norm{u}_{H^{k+1}(\Gamma)}
  \norm{\phi}_{H^4(\Gamma)}.
\end{align*}
Combining the preceding estimates proves \eqref{eq:dual-interpolation-consistency}.
\end{proof}

\begin{lemma}[Consistency estimates used in the duality argument]
\label{lem:L2-consistency}
Let $u\in H^{s_k}_{\#}(\Gamma)$ solve \eqref{eq:strong-problem} and let $u_h$ solve \eqref{eq:discrete-problem}. 
Define 
\begin{equation}\label{eq:energy-error-short}
  M_h(u)
  :=
  h^{k-1}\norm{u}_{H^{k+1}(\Gamma)}
  +
  h^r\norm{u}_{H^3(\Gamma)}
  +
  h^{r+1}\norm{u}_{H^4(\Gamma)}.
\end{equation}
Define $e_h:=u^e-u_h$.
If $\phi\in H^4(\Gamma)$ and
$\phi_h=I_h^\#\phi^e$, then
\begin{align}
  \left|
    (\deltag^2\phi,e_h^l)_\Gamma
    -
    a_h(e_h,\phi^e)
  \right|
  &\le
  C(h^{\alpha_k}+h^r)
  M_h(u)
  \norm{\phi}_{H^4(\Gamma)},
  \label{eq:dual-consistency-defect}
  \\
  \left|
    a_h(u^e,\phi_h)-\ell_h(\phi_h)
  \right|
  &\le
  C\left(
    h^{q_k}+h^{r+1}
  \right)
  \norm{u}_{H^{s_k}(\Gamma)}
  \norm{\phi}_{H^4(\Gamma)}.
  \label{eq:dual-tested-primal-consistency}
\end{align}
\end{lemma}

\begin{proof}
We only sketch the proof here, since the technical details are quite similar to those in the lemmas in Sections~\ref{sec:auxiliary-est} and~\ref{sec:proj-consistency} and Lemma \ref{lem:dual-interpolation-consistency}; we emphasize only the key points.
 
Splitting $e_h=(u^e-I_h^\#u^e)+(I_h^\#u^e-u_h)$ and using the interpolation estimate \eqref{eq:interpolation-energy}, the energy estimate $\enorm{I_h^\#u^e-u_h}\le CM_h(u)$ (by Theorem~\ref{thm:consistency-error} and coercivity of $a_h$ in Lemma~\ref{lem:stability}), and the inverse inequality to $D_{\Gamma_h}^2(I_h^\#u^e-u_h)$ lead to  
\begin{equation}\label{eq:L2-error-auxiliary-control}
   \mathfrak X_h(e_h):=\enorm{e_h}
  +
  \left(
    \sum_{K\in\Kh}
    h_K^2
    \norm{D_{\Gamma_h}^2e_h}_{L^2(K)}^2
  \right)^{1/2}
  \le
  CM_h(u).
\end{equation}

For \eqref{eq:dual-consistency-defect}, use the exact broken identity
\[
  (\deltag^2\phi,e_h^l)_\Gamma
  =
  \sum_{K^l\in\mathcal K_h^l}
  (\deltag e_h^l,\deltag\phi)_{K^l}
  -
  \sum_{E^l\in\mathcal E_h^l}
  (J_{E^l}^{\,l}(e_h^l),\deltag\phi)_{E^l}
\]
and compare it term by term with the expanded form
\eqref{eq:expanded-form} of $a_h(e_h,\phi^e)$.
The element-operator, conormal, surface-measure, and edge-measure
comparisons are bounded by
\[
  Ch^r\mathfrak X_h(e_h)\norm{\phi}_{H^4(\Gamma)}.
\]
The difference between $\deltag\phi$ and its $Q_h$-projection
contributes
\[
  Ch^{\alpha_k}
  \mathfrak X_h(e_h)
  \norm{\phi}_{H^4(\Gamma)},
\]
as $\deltag\phi$ possesses at most two derivatives under the assumed $H^4$-regularity of $\phi$ and the $Q_h$-projection estimate gives an error $\mathcal{O}(h^{\alpha_k})$.   
The terms containing $J_E(\phi^e)$, including the lifting-lifting and penalty terms, are $\mathcal{O}(h^r)\mathfrak X_h(e_h)$ by \eqref{eq:smooth-geometric-jump-global} and lifting stability. 
Together with \eqref{eq:L2-error-auxiliary-control}, this proves \eqref{eq:dual-consistency-defect}.

For \eqref{eq:dual-tested-primal-consistency}, set $\eta_h:=\phi_h-\phi^e$ and write
\begin{align*}
  a_h(u^e,\phi_h)-\ell_h(\phi_h)
  ={}&
  \bigl[
    a_h(u^e,\eta_h)-(f,\eta_h^l)_\Gamma
  \bigr]
  +
  \bigl[
    a_h(u^e,\phi^e)-a(u,\phi)
  \bigr]
  +
  \bigl[
    (f,\phi_h^l)_\Gamma-\ell_h(\phi_h)
  \bigr].
\end{align*}
The first term is bounded by $ C\left(h^{k-1+\alpha_k}+h^{r+\alpha_k}\right)\norm{u}_{H^{s_k}(\Gamma)}\norm{\phi}_{H^4(\Gamma)}$, due to \eqref{eq:dual-interpolation-consistency}.
The second term is bounded by $\mathcal{O}(h^{r+1})$ using Lemma~\ref{lem:smooth-smooth-consistency} (smooth-smooth consistency estimate).
The last term is also $\mathcal{O}(h^{r+1})$ by the load consistency estimate in Lemma~\ref{lem:known-load-consistency}. Since
\[
  k-1+\alpha_k=q_k,
  \qquad
  r+\alpha_k\ge r+1,
\]
estimate \eqref{eq:dual-tested-primal-consistency} follows.
\end{proof}

We emphasize that the estimate \eqref{eq:smooth-smooth-consistency} proved in Lemma~\ref{lem:smooth-smooth-consistency} is crucial to retain the order $\mathcal{O}(h^{r+1})$ in the estimate \eqref{eq:dual-tested-primal-consistency}. A previous comparison in Section~\ref{sec:proj-consistency} between $a_h$ and $a$ yields only an $\mathcal{O}(h^{r})$ error, since only discrete test functions were used; see, for instance, \eqref{eq:direct-element-comparison}. The use of $z,\phi \in H^4(\Gamma)$ as arguments in \eqref{eq:smooth-smooth-consistency} leads to the increased order.

Now we consider a dual problem as follows. 
Recall that $\Gamma\in C^{r+2}\cap C^{s_k,1}$ is a connected, compact surface without boundary. 
For every $\psi\in L^2(\Gamma)$ with $(\psi,1)_\Gamma=0$, there exists $\phi\in H^4(\Gamma)$ that solves 
\begin{equation}\label{eq:L2-dual-problem}
  \deltag^2\phi=\psi,
  \qquad
  (\phi,1)_\Gamma=0,
\end{equation}
and satisfies 
\begin{equation}\label{eq:L2-dual-regularity}
  \norm{\phi}_{H^4(\Gamma)}
  \le
  C\norm{\psi}_{L^2(\Gamma)}.
\end{equation}

\begin{theorem}[$L^2$-error estimate]
\label{thm:L2-error}
Let $u\in H^{s_k}(\Gamma)$ solve \eqref{eq:strong-problem}, and let $u_h\in W_{h,\#}$ solve \eqref{eq:discrete-problem}. 
Then
\begin{equation}\label{eq:L2-error-lifted-data}
  \norm{u-u_h^l}_{L^2(\Gamma)/\mathbb R}
  \le
  C\left(
    h^{q_k}+h^{r+1}
  \right)
  \norm{f}_{H^{s_k-4}(\Gamma)}.
\end{equation}
\end{theorem}

\begin{proof}
Set $e_h:=u^e-u_h$ and $\psi:=e_h^l-|\Gamma|^{-1}(e_h^l,1)_{\Gamma}$, and let $\phi$ solve \eqref{eq:L2-dual-problem}. 
Then
\[
  \norm{\psi}_{L^2(\Gamma)}^2
  =
  (\deltag^2\phi,e_h^l)_\Gamma.
\]
Let $\phi_h=I_h^\#\phi^e$. Adding and subtracting the discrete form and
using the discrete problem give the exact decomposition
\begin{align}
  \norm{\psi}_{L^2(\Gamma)}^2
  ={}&
  a_h(e_h,\phi^e-\phi_h)
 +
  \left[
    a_h(u^e,\phi_h)-\ell_h(\phi_h)
  \right]
  +
  \left[
    (\deltag^2\phi,e_h^l)_\Gamma
    -
    a_h(e_h,\phi^e)
  \right].
  \label{eq:L2-master-identity}
\end{align}
By continuity of $a_h$, Theorem~\ref{thm:energy}, and the interpolation estimate $\enorm{\phi^e-\phi_h}\le Ch^{\alpha_k}\norm{\phi}_{H^4(\Gamma)}$,
\[
  |a_h(e_h,\phi^e-\phi_h)|
  \le
  Ch^{\alpha_k}
  M_h(u)
  \norm{\phi}_{H^4(\Gamma)}.
\]
The remaining two terms are bounded by
Lemma~\ref{lem:L2-consistency}. Therefore,
\[
  \norm{\psi}_{L^2(\Gamma)}^2
  \le
  C\left(
    h^{q_k}+h^{r+1}
  \right)
  \norm{u}_{H^{s_k}(\Gamma)}
  \norm{\phi}_{H^4(\Gamma)}.
\]
Using \eqref{eq:L2-dual-regularity} and the regularity estimate \eqref{eq:higher-regularity} proves \eqref{eq:L2-error-lifted-data}.
\end{proof}

\begin{corollary}[$L^2$ error on the computational surface]
\label{cor:L2-error-computational-surface}
Under the assumptions of Theorem~\ref{thm:L2-error},
\begin{equation}\label{eq:L2-error-gammah}
  \norm{u^e-u_h}_{L^2(\Gamma_h)}
  \le
  C\left(
    h^{q_k}+h^{r+1}
  \right)
  \norm{f}_{H^{s_k-4}(\Gamma)}.
\end{equation}
\end{corollary}

\begin{proof}
Since $(u,1)_\Gamma=0$ and $(u_h,1)_{\Gamma_h}=0$, changing variables gives
\[
  (u-u_h^l,1)_\Gamma
  =
  -((\mu_h-1)u_h,1)_{\Gamma_h}.
\]
Hence, by \eqref{eq:mu-errors}, $\left||\Gamma|^{-1}(u-u_h^l,1)_\Gamma\right|$ is of order $\mathcal{O}(h^{r+1})$.
Combining it with Theorem~\ref{thm:L2-error} controls the full $L^2(\Gamma)$ norm of $u-u_h^l$. The norm equivalence in Lemma~\ref{lem:norm-trace} then gives \eqref{eq:L2-error-gammah}.
\end{proof}

%=======================================================================
\section{Numerical experiments}
%=======================================================================

\subsection{Assembly on a parametric surface}
\label{sec:implementation}

We describe the algebraic assembly of \eqref{eq:discrete-problem}.  Let
$\{\phi_i\}$ and $\{\psi_a\}$ be bases of $W_h$ and $Q_h$, respectively, and
define
\begin{align}
  (K_\Delta)_{ij}
  &:=(\deltagh\phi_i,\deltagh\phi_j)_{\Gammah},
  \label{eq:implementation-K}
  \\
  D_{ai}
  &:=(\psi_a,\deltagh\phi_i)_{\Gammah},
  \label{eq:implementation-D}
  \\
  L_{ai}
  &:=\sum_{E\in\Eh}(\mean{\psi_a},J_E(\phi_i))_E,
  \label{eq:implementation-L}
  \\
  S_{ij}
  &:=\sum_{E\in\Eh}h_E^{-1}
  (J_E(\phi_i),J_E(\phi_j))_E.
  \label{eq:implementation-S}
\end{align}
Let $M_Q$ be the mass matrix of $Q_h$.  By \eqref{eq:lifting}, the coefficient
vector of $R_hJ(v_h)$ is $M_Q^{-1}L\mathbf v$.  Expanding the square in
\eqref{eq:discrete-form} gives the stiffness matrix
\begin{equation}\label{eq:implementation-stiffness-curved}
  A_h
  =K_\Delta
   -D^TM_Q^{-1}L
   -L^TM_Q^{-1}D
   +L^TM_Q^{-1}L
   +\beta S.
\end{equation}
The inverse of $M_Q$ is local because $Q_h$ is discontinuous. Indeed, if $\mathbf r_v=M_Q^{-1}L\mathbf v$ denotes the coefficient
vector of $R_hJ(v_h)$, then
\[
  \bigl(R_hJ(u_h),R_hJ(v_h)\bigr)_{\Gamma_h}
  =\mathbf r_u^T M_Q\mathbf r_v
  =\mathbf u^T L^TM_Q^{-1}L\mathbf v,
\]
where we have used the symmetry of $M_Q$.  

Formula \eqref{eq:implementation-stiffness-curved} is valid for every geometry
degree $r$ and every solution degree $k\ge2$.  On an affine mesh ($r=1$),
$\deltagh W_h\subset Q_h$ and
$K_\Delta=D^TM_Q^{-1}D$, so that
\begin{equation}\label{eq:implementation-stiffness-flat}
  A_h=(D-L)^TM_Q^{-1}(D-L)+\beta S.
\end{equation}
However, on a curved mesh (with general $r$), replacing $K_\Delta$ by $D^TM_Q^{-1}D$ would ignore the component $(I-\Pi_{Q_h})\deltagh v_h$. 

The two conormals on every curved interior edge must be computed from their respective element parametrizations and inserted in \eqref{eq:JE}.  All edge expressions are invariant under swapping the labels $+$ and $-$.  Since the metric coefficients, conormals, and element Laplacians are generally nonpolynomial, sufficiently accurate element and edge quadrature is required.
We use NGSolve \cite{schoberl2014c++} to implement our method. 
When $r \ge 2$, the computational mesh is generated using the command \texttt{mesh.Curve(r)} in NGSolve. For refinement, we first uniformly refine the affine mesh $\gamma_h$ and then apply \texttt{mesh.Curve(r)} to generate the corresponding curved refined mesh.

The matrix has constants in its kernel.  One may impose the discrete mean by
a Lagrange multiplier, work on the quotient space, or fix one algebraic degree
of freedom and subtract the discrete mean afterward.  

\subsection{Convergence rate}
\label{sec:numerical-experiment}

We validate the method using the sphere benchmark in \cite[Section~6]{larsson2017continuous}.  Let $\Gamma=\mathbb S^2$ and set
\begin{equation}\label{eq:sphere-exact-solution}
  u(x,y,z)=3x^2y-y^3,
  \qquad (x,y,z)\in\mathbb S^2.
\end{equation}
This is a spherical harmonic of degree three, and hence
\begin{equation}\label{eq:sphere-eigenvalue}
  -\deltag u=12u,
  \qquad
  f=\deltag^2u=144u.
\end{equation}
The closest-point extension used to evaluate errors on $\Gammah$ is
\begin{equation}\label{eq:sphere-extension}
  u^e(x,y,z)
  =\frac{3x^2y-y^3}{(x^2+y^2+z^2)^{3/2}}.
\end{equation} 

We report the $L^2$-error 
\begin{equation}\label{eq:numerical-L2-error}
  E_{0,h}:=\norm{u^e-u_h}_{L^2(\Gammah)},
\end{equation}
and the energy error
\begin{equation}\label{eq:numerical-energy-error}
  E_{2,h}:=(\norm{\deltagh(u^e-u_h)}_{L^2(\Gammah)}^2+\sum_{E\in\Eh}h_E^{-1}
  \norm{J_E(u^e-u_h)}_{L^2(E)}^2)^{1/2}.
\end{equation}
We also measure
\begin{equation}\label{eq:numerical-reconstructed-error}
  E_{R,h}
  :=\norm{\tdeltagh u_h-(\deltag u)^e}_{L^2(\Gammah)}.
\end{equation}

Table~\ref{tab:sphere-convergence} gives the errors and observed rates.  
The rates are calculated by
\begin{equation*}
  \operatorname{rate}
  =\frac{\log(E_{h_{j-1}}/E_{h_j})}
  {\log(h_{j-1}/h_j)}.
\end{equation*}

\begin{table}[htbp]
\centering
\small
\caption{Convergence for the reconstructed-Laplacian CDG method on the unit sphere with $k=2$, $r=1$, $\ell=0$, and $\beta=1$.}
\label{tab:sphere-convergence}
\begin{tabular}{r|r|r|c|cc|cc|cc}
\hline
level & $N_T$ & dofs & $h$
& $E_{0,h}$ & rate
& $E_{2,h}$ & rate
& $E_{R,h}$ & rate\\
\hline
0 &    108 &    218 & $8.554\mathrm e{-1}$
  & $3.849\mathrm e{-1}$ & -
  & $1.891\mathrm e{1}$  & -
  & $7.536\mathrm e{0}$  & -\\
1 &    432 &    866 & $4.522\mathrm e{-1}$
  & $1.392\mathrm e{-1}$ & 1.60
  & $6.738\mathrm e{0}$  & 1.62
  & $3.862\mathrm e{0}$  & 1.05\\
2 &   1728 &   3458 & $2.296\mathrm e{-1}$
  & $3.887\mathrm e{-2}$ & 1.88
  & $2.718\mathrm e{0}$  & 1.34
  & $1.952\mathrm e{0}$  & 1.01\\
3 &   6912 &  13826 & $1.153\mathrm e{-1}$
  & $1.007\mathrm e{-2}$ & 1.96
  & $1.234\mathrm e{0}$  & 1.15
  & $9.804\mathrm e{-1}$ & 1.00\\
4 &  27648 &  55298 & $5.768\mathrm e{-2}$
  & $2.549\mathrm e{-3}$ & 1.99
  & $5.954\mathrm e{-1}$ & 1.05
  & $4.912\mathrm e{-1}$ & 1.00\\
\hline
\end{tabular}
\end{table}

\begin{table}[htbp]
\centering
\small
\caption{Convergence for the reconstructed-Laplacian CDG method on the unit sphere with $k=3$, $r=2$, $\ell=1$, and $\beta=1$.}
\label{tab:sphere-convergence-2}
\begin{tabular}{r|r|r|c|cc|cc|cc}
\hline
level & $N_T$ & dofs & $h$
& $E_{0,h}$ & rate
& $E_{2,h}$ & rate
& $E_{R,h}$ & rate\\
\hline
0 &    108 &    488 & $8.554\mathrm e{-1}$
  & $3.361\mathrm e{-2}$ & -
  & $5.416\mathrm e{0}$  & -
  & $1.424\mathrm e{0}$  & -\\
1 &    432 &   1946 & $4.522\mathrm e{-1}$
  & $3.244\mathrm e{-3}$ & 3.67
  & $1.489\mathrm e{0}$  & 2.03
  & $3.809\mathrm e{-1}$ & 2.07\\
2 &   1728 &   7778 & $2.296\mathrm e{-1}$
  & $2.410\mathrm e{-4}$ & 3.84
  & $3.777\mathrm e{-1}$ & 2.02
  & $9.677\mathrm e{-2}$ & 2.02\\
3 &   6912 &  31106 & $1.153\mathrm e{-1}$
  & $1.643\mathrm e{-5}$ & 3.90
  & $9.500\mathrm e{-2}$ & 2.00
  & $2.426\mathrm e{-2}$ & 2.01\\
4 &  27648 & 124418 & $5.768\mathrm e{-2}$
  & $1.071\mathrm e{-6}$ & 3.95
  & $2.380\mathrm e{-2}$ & 2.00
  & $6.067\mathrm e{-3}$ & 2.00\\
\hline
\end{tabular}
\end{table}

The observed energy error and reconstructed-Laplacian rates agree with Theorem~\ref{thm:energy} and Corollary~\ref{cor:reconstructed-laplacian}, respectively, for $k=2$, $r=1$ in Table \ref{tab:sphere-convergence} and $k=3$, $r=2$ in Table~\ref{tab:sphere-convergence-2}. For the $L^2$ error, we compare the data with theoretical result in Corollary~\ref{cor:L2-error-computational-surface}. Note that when $k=2$ and $r=1$, $q_k=r+1=2$, and the $L^2$-rate in Table \ref{tab:sphere-convergence} matches it. When $k=3$ and $r=2$, we have $q_k=4$ and $r+1=3$, but we do observe an $\mathcal{O}(h^4)$ convergence in $L^2$-norm in Table~\ref{tab:sphere-convergence-2}. This is due to a superconvergence in geometric quantities as shown in Table~\ref{tab:sphere-geometric-errors} when $r=2$. The reason could be the symmetry in unit sphere and the meshes; we refer to \cite{hardering2026odd} for a detailed discussion on such superconvergence.      
The same one-order superconvergence of the geometric approximation also explains the approximately $\mathcal{O}(h^3)$ energy convergence observed for $k=4$ and $r=2$ in Table~\ref{tab:sphere-convergence-3}; for a general second-order geometric approximation, Theorem~\ref{thm:energy} guarantees only $\mathcal{O}(h^2)$ energy convergence.

Table~\ref{tab:sphere-geometric-errors} provides numerical evidence for selected consequences of Assumption~\ref{lem:parametric-geometry} on the tested quadratic mesh family.  
At each refinement level, the distance, normal, and surface and edge measure defects are evaluated at sufficiently dense sampling points on the actual curved elements and edges, using the exact sphere geometry and the computational element transformations.  
The largest sampled values are reported as approximations of the corresponding $L^\infty$-norms. 
The observed rates support the predicted geometric orders for this mesh sequence.

Moreover, we emphasize that the reconstructed-Laplacian CDG formulation requires no tuning of the stabilization parameter~$\beta$. Taking $\beta = 1$, the results in Table~\ref{tab:sphere-convergence} exhibit the expected stability and convergence, whereas in the same experiment the IPDG method of~\cite{larsson2017continuous} shows numerical instability with the same value of~$\beta$.

\begin{table}[htbp]
\centering
\small
\caption{Geometric approximation errors in ${L^\infty(\Gamma_h)}$ or $L^\infty(\mathcal E_h)$ for geometric quantities in Assumption \ref{lem:parametric-geometry} with $r=2$.}
\label{tab:sphere-geometric-errors}
\begin{tabular}{r|c|cc|cc|cc|cc}
\hline
level & $h$
& $d$ & rate
& $\vn^e-\vn_h$ & rate
& $1-\mu_h$ & rate
& $1-\mu_{h,E}$ & rate\\
\hline
0 & $8.55\mathrm e{-1}$
  & $4.468\mathrm e{-3}$ & -
  & $3.987\mathrm e{-2}$ & -
  & $8.987\mathrm e{-3}$ & -
  & $2.438\mathrm e{-3}$ & -\\
1 & $4.52\mathrm e{-1}$
  & $3.225\mathrm e{-4}$ & 4.12
  & $5.882\mathrm e{-3}$ & 3.00
  & $6.450\mathrm e{-4}$ & 4.13
  & $1.722\mathrm e{-4}$ & 4.16\\
2 & $2.30\mathrm e{-1}$
  & $2.096\mathrm e{-5}$ & 4.03
  & $7.705\mathrm e{-4}$ & 3.00
  & $4.191\mathrm e{-5}$ & 4.03
  & $1.115\mathrm e{-5}$ & 4.04\\
3 & $1.15\mathrm e{-1}$
  & $1.323\mathrm e{-6}$ & 4.01
  & $9.748\mathrm e{-5}$ & 3.00
  & $2.646\mathrm e{-6}$ & 4.01
  & $7.031\mathrm e{-7}$ & 4.01\\
4 & $5.77\mathrm e{-2}$
  & $8.289\mathrm e{-8}$ & 4.00
  & $1.222\mathrm e{-5}$ & 3.00
  & $1.658\mathrm e{-7}$ & 4.00
  & $4.405\mathrm e{-8}$ & 4.00\\
\hline
\end{tabular}
\end{table}

\begin{table}[htbp]
\centering
\small
\caption{Convergence for the reconstructed-Laplacian CDG method on the unit sphere with $k=4$, $r=2$, $\ell=2$, and $\beta=1$.}
\label{tab:sphere-convergence-3}
\begin{tabular}{r|r|r|c|cc|cc|cc}
\hline
level & $N_T$ & dofs & $h$
& $E_{0,h}$ & rate
& $E_{2,h}$ & rate
& $E_{R,h}$ & rate\\
\hline
0 &    108 &     866 & $8.554\mathrm e{-1}$
  & $1.859\mathrm e{-2}$ & -
  & $9.636\mathrm e{-1}$ & -
  & $1.953\mathrm e{-1}$ & -\\
1 &    432 &    3458 & $4.522\mathrm e{-1}$
  & $1.186\mathrm e{-3}$ & 4.32
  & $8.364\mathrm e{-2}$ & 3.83
  & $2.384\mathrm e{-2}$ & 3.30\\
2 &   1728 &   13826 & $2.296\mathrm e{-1}$
  & $7.465\mathrm e{-5}$ & 4.08
  & $7.361\mathrm e{-3}$ & 3.59
  & $2.951\mathrm e{-3}$ & 3.08\\
3 &   6912 &   55298 & $1.153\mathrm e{-1}$
  & $4.675\mathrm e{-6}$ & 4.02
  & $6.810\mathrm e{-4}$ & 3.45
  & $3.667\mathrm e{-4}$ & 3.03\\
4 &  27648 &  221186 & $5.768\mathrm e{-2}$
  & $2.928\mathrm e{-7}$ & 4.00
  & $6.834\mathrm e{-5}$ & 3.32
  & $4.570\mathrm e{-5}$ & 3.01\\
\hline
\end{tabular}
\end{table}

\section{Extensions to Laplacian-dominated nonlinear surface problems}
\label{sec:nonlinear-extensions}
The reconstructed-Laplacian construction is not restricted to the linear surface biharmonic problem.  
It applies directly to surface PDEs whose principal part is governed by the Laplace-Beltrami operator.  
A representative class is the $L^2$-gradient flow equation of the form 
\begin{equation}\label{eq:general-laplacian-gradient-flow}
  \partial_t u
  +
  \gamma\deltag^2u
  -
  \eta\deltag u
  +
  \Psi'(u)
  =
  f
  \qquad\text{on }\Gamma,
\end{equation}
where $\gamma>0$, $\eta\in\mathbb R$, and $\Psi$ is a possibly nonconvex potential.  The associated energy is
\begin{equation}\label{eq:general-laplacian-energy}
  \mathcal F(u)
  :=
  \frac{\gamma}{2}
  \norm{\deltag u}_{L^2(\Gamma)}^2
  +
  \frac{\eta}{2}
  \norm{\nabla_\Gamma u}_{L^2(\Gamma)}^2
  +
  \int_\Gamma\Psi(u)\,ds
  -
  (f,u)_\Gamma,
\end{equation}
where $u\in H^2(\Gamma)$, $f\in L^2(\Gamma)$. Moreover, we assume $\Psi\in C^{\infty}$, $\Psi^{\prime}$ is locally Lipschitz and $\Psi$ satisfies the following condition: for some $p>2$,
\begin{equation}\label{eq:potential-lower-growth}
  \Psi(s)
  \ge
  c_p|s|^p-c_0,
  \qquad s\in\mathbb R,
\end{equation}
where $c_p,c_0>0$.

This class includes semilinear biharmonic equations, the surface analogue of the extended Fisher--Kolmogorov equation, and, after expanding a shifted Laplace-Beltrami operator, the surface Swift--Hohenberg equation.  
The purpose of this section is not to develop a complete nonlinear or temporal error analysis.  
Instead, we discuss the stability mechanism inherited from the reconstructed Laplacian and demonstrate the resulting method on the Swift-Hohenberg equation as an example.

\subsection{Stability inherited from the reconstructed Laplacian}
\label{subsec:nonlinear-coercivity}
We denote the stabilization term as $S_h(v_h,w_h):=\sum_{E\in\Eh}h_E^{-1}(J_E(v_h),J_E(w_h))_E$ and denote the leading reconstructed bilinear form as 
\begin{equation}\label{eq:leading-reconstructed-form}
  a_h^{(0)}(v_h,w_h)
  :=
  \bigl(
    \tdeltagh v_h,
    \tdeltagh w_h
  \bigr)_{\Gammah}
  +
  \beta S_h(v_h,w_h),
  \qquad
  v_h,w_h\in W_h.
\end{equation}
Here $\tdeltagh$ is the reconstructed surface Laplacian discussed in previous sections. By Lemma \ref{lem:stability}, we have the stability  
\begin{equation}\label{eq:any-beta-equivalence}
  c_\beta \enorm{v_h}^2
  \le
  a_h^{(0)}(v_h,v_h)
  \le
  C_\beta \enorm{v_h}^2
  \qquad
  \forall v_h\in W_h.
\end{equation}  

We define the discrete energy for $v_h\in W_h$
\begin{equation}\label{eq:general-discrete-nonlinear-energy}
  \mathcal F_h(v_h)
  :=
  \frac{\gamma}{2}a_h^{(0)}(v_h,v_h)
  +
  \frac{\eta}{2}
  \norm{\nabla_{\Gammah}v_h}_{L^2(\Gammah)}^2
  +
  \int_{\Gammah}\Psi(v_h)\,ds_h
  -
  \ell_h(v_h).
\end{equation}
We assume $\ell_h$ is a consistent discretization of the load and is uniformly bounded such that $ |\ell_h(v_h)|\le C_\ell\norm{v_h}_{L^2(\Gammah)}$.

\begin{proposition}[Coercivity of the total discrete energy]
\label{prop:nonlinear-energy-coercivity}
Let $\gamma>0$ and let $\beta>0$ be fixed.  There are constants
$c_{\mathcal F}>0$ and $C_{\mathcal F}>0$, independent of $h$, such
that
\begin{equation}\label{eq:nonlinear-energy-coercivity}
  \mathcal F_h(v_h)+C_{\mathcal F}
  \ge
  c_{\mathcal F}
  \left(
    \enorm{v_h}^2
    +
    \norm{v_h}_{L^2(\Gammah)}^2
    +
    \norm{v_h}_{L^p(\Gammah)}^p
  \right)
  \qquad
  \forall v_h\in W_h.
\end{equation}
The constant $c_{\mathcal F}$ may depend on $\beta$, but the result is
valid for every fixed $\beta>0$.
\end{proposition}

\begin{proof}
We first note that elementwise integration by parts, the
trace inequality, the lifting estimate, and Young's inequality give,
for every $\epsilon>0$,
\begin{equation}\label{eq:discrete-H1-H2-interpolation}
  \norm{\nabla_{\Gammah}v_h}_{L^2(\Gammah)}^2
  \le
  \epsilon \enorm{v_h}^2
  +
  C_\epsilon
  \norm{v_h}_{L^2(\Gammah)}^2.
\end{equation}

If $\eta<0$, apply \eqref{eq:discrete-H1-H2-interpolation} and choose $\epsilon>0$ sufficiently small relative to $\gamma c_\beta$, where $c_\beta$ is the coercive constant in \eqref{eq:any-beta-equivalence}.  
If $\eta\ge0$, the tangential gradient term is nonnegative and no absorption is needed.  
This gives
\[
  \frac{\gamma}{2}a_h^{(0)}(v_h,v_h)
  +
  \frac{\eta}{2}
  \norm{\nabla_{\Gammah}v_h}_{L^2(\Gammah)}^2
  \ge
  c\enorm{v_h}^2
  -
  C\norm{v_h}_{L^2(\Gammah)}^2.
\]
Since $p>2$, applying Young's inequality we obtain 
\[
  C\norm{v_h}_{L^2(\Gammah)}^2
  +
  |\ell_h(v_h)|
  \le
  \frac{c_p}{2}
  \norm{v_h}_{L^p(\Gammah)}^p
  +
  C.
\]
Combining these estimates with
\eqref{eq:potential-lower-growth} proves
\eqref{eq:nonlinear-energy-coercivity}.
\end{proof}

Proposition~\ref{prop:nonlinear-energy-coercivity} explains the role of the biharmonic principal part in nonlinear problems.  
The reconstructed form controls the discrete second derivatives for any fixed $\beta>0$, lower-order derivative terms are relatively bounded, and the nonlinear potential supplies any remaining $L^2$ control. 
Consequently, the discrete energy is coercive.  
This, together with the smoothness of $\Psi$, ensures the existence of discrete stationary solutions by the direct method of calculus of variations, and together with the energy dissipation structure, yields global-in-time existence of the semidiscrete solution of gradient flows.

The semidiscrete approximation of \eqref{eq:general-laplacian-gradient-flow} reads: find $u_h(t)\in W_h$ such that
\begin{equation}\label{eq:general-nonlinear-semidescrete}
  \bigl(\partial_tu_h,v_h\bigr)_{\Gammah}
  +
  \gamma a_h^{(0)}(u_h,v_h)
  +
  \eta
  \bigl(
    \nabla_{\Gammah}u_h,
    \nabla_{\Gammah}v_h
  \bigr)_{\Gammah}
  +
  \bigl(\Psi'(u_h),v_h\bigr)_{\Gammah}
  =
  \ell_h(v_h)
\end{equation}
for all $v_h\in W_h$.  
For a time-independent load, testing with $v_h=\partial_tu_h$ gives the energy dissipation relation: 
\begin{equation}\label{eq:general-nonlinear-energy-law}
  \frac{d}{dt}\mathcal F_h(u_h(t))=-\norm{\partial_tu_h}_{L^2(\Gammah)}^2.
\end{equation}

A complete nonlinear error analysis would combine the approximation and geometric consistency estimates proved in the preceding sections with local Lipschitz estimates for $\Psi'$ and a Gronwall argument.
We expect that the same coercivity estimate supplies a principal bound for energy-stable implicit or convex-splitting time discretizations.
Since the focus of the present work is the spatial reconstructed-Laplacian construction, a complete nonlinear and temporal error analysis for \eqref{eq:general-laplacian-gradient-flow} is not pursued here.

\subsection{Surface Swift-Hohenberg equation}
\label{subsec:swift-hohenberg}
We consider the surface Swift-Hohenberg equation as an example.
The equation reads, with $\varrho\ge 0$, 
\begin{equation}\label{eq:surface-swift-hohenberg}
  \partial_tu
  +
  \bigl(\deltag+\kappa^2\bigr)^2u
  -
  \varrho u
  +
  u^3
  =
  g
  \qquad
  \text{on }\Gamma\times(0,T].
\end{equation}
The Swift-Hohenberg equation is a classical model of wavelength selection and pattern formation \cite{swift1977hydrodynamic,cross1993pattern}.  
Numerical methods for this model have recently been considered in \cite{yu2023stabilized}.

For a time-independent load, equation
\eqref{eq:surface-swift-hohenberg} is the $L^2$-gradient flow of
\begin{equation}\label{eq:continuous-swift-energy}
  \mathcal F_{\mathrm{SH}}(v)
  :=
  \frac12
  \norm{
    \bigl(\deltag+\kappa^2\bigr)v
  }_{L^2(\Gamma)}^2
  -
  \frac{\varrho}{2}
  \norm{v}_{L^2(\Gamma)}^2
  +
  \frac14
  \norm{v}_{L^4(\Gamma)}^4
  -
  (g,v)_\Gamma.
\end{equation}
Expanding the first term by integration by parts gives
\begin{align*}
  \frac12
  \norm{
    \bigl(\deltag+\kappa^2\bigr)v
  }_{L^2(\Gamma)}^2
  &=
  \frac12\norm{\deltag v}_{L^2(\Gamma)}^2
  -
  \kappa^2
  \norm{\nabla_\Gamma v}_{L^2(\Gamma)}^2
  +
  \frac{\kappa^4}{2}
  \norm{v}_{L^2(\Gamma)}^2.
\end{align*}
Thus the equation belongs to the class \eqref{eq:general-laplacian-gradient-flow}, with a negative tangential gradient contribution and a quartic potential.

In the computation, we define the reconstructed shifted Laplacian $B_hv_h:=\tdeltagh v_h+\kappa^2v_h$.
We use the first-order convex-splitting scheme: given $u_h^{n-1}\in W_h$, find $u_h^n\in W_h$ such that
\begin{align}
  \left(
    \frac{u_h^n-u_h^{n-1}}{\tau},
    v_h
  \right)_{\Gammah}
  &+
  \bigl(B_hu_h^n,B_hv_h\bigr)_{\Gammah}
  +
  \beta S_h(u_h^n,v_h)
  \notag\\
  &+
  \bigl((u_h^n)^3,v_h\bigr)_{\Gammah}
  -
  \varrho(u_h^{n-1},v_h)_{\Gammah}
  =
  (g^{e}(t_n),v_h)_{\Gammah}.
  \label{eq:swift-convex-splitting}
\end{align}
We emphasize that, compared to the surface biharmonic equation on the closed surface $\Gamma$, no zero-mean constraint is imposed, and the load is not corrected to have zero mean.
We also note that \eqref{eq:swift-convex-splitting} is a nonlinear equation to solve at each step, and we use a Newton method to solve it in the computation.    

To test the spatial accuracy in the evolution problem, we integrate from $t=0$ to $T=1$ and compare $u_h(T)$ with the manufactured solution $u^e(T)$, using a sufficiently small time step to limit contamination by the first-order temporal error.

Take $\Gamma=\mathbb S^2$, $\kappa^2=1$, $ \varrho=1$.   
We choose
\begin{equation}
  u(x,y,z,t)=e^{-t}xy.
  \label{eq:swift-time-manufactured-solution}
\end{equation}
Since $\partial_tu=-e^{-t}xy=-u$ and $\deltag(xy)=-6xy$ on $\mathbb S^2$, we have $\bigl(\deltag+1\bigr)u=-5u$ and therefore $\bigl(\deltag+1\bigr)^2u=25u$. 
It follows that the corresponding forcing term is
\begin{equation}
  g(x,y,z,t)
  =
  23e^{-t}xy
  +
  e^{-3t}x^3y^3.
  \label{eq:swift-time-manufactured-forcing}
\end{equation}
The initial condition is therefore $u(x,y,z,0)=xy$. The extensions of these functions are defined by closest-point projection, which is $\vx/|\vx|$ for the unit sphere.  

We fix a physical time $T=1$ in the experiment, and test the spatial discretization error at $t=T$.  
Table~\ref{tab:swift-fixed-time-spatial} reports the results at $T=1$. As $r=2$, $k=3$, the observed rates approach the expected orders $\mathcal{O}(h^3)$ and $\mathcal{O}(h^2)$ for $L^2$ norm and the energy norm error respectively. In particular, we choose a small time step $\tau=0.001$ to prevent the first-order time discretization error from masking the high-order spatial convergence. 
We emphasize that observing the optimal spatial convergence order for smaller $h$ would require a substantially smaller time step $\tau$; although a high-order temporal discretization would be desirable, it is beyond the scope of this work.

\begin{table}[htbp]
\centering
\small
\caption{Spatial convergence for the reconstructed-Laplacian discretization of the surface Swift-Hohenberg equation. $T=1$, $r=2$, $k=3$, $\ell=1$, $\beta=1$, $\tau=0.001$.}
\label{tab:swift-fixed-time-spatial}
\begin{tabular}{r|r|r|c|cc|cc}
\hline
level & $N_T$ & dofs & $h$
& $E_{0,h}$ & rate
& $E_{2,h}$ & rate\\
\hline
0 &   20 &   92 & $1.051\mathrm e{0}$
  & $1.349\mathrm e{-2}$ & -
  & $6.715\mathrm e{-1}$ & -\\
1 &   80 &  362 & $6.180\mathrm e{-1}$
  & $2.493\mathrm e{-3}$ & 3.18
  & $3.079\mathrm e{-1}$ & 1.47\\
2 &  320 & 1442 & $3.249\mathrm e{-1}$
  & $2.513\mathrm e{-4}$ & 3.57
  & $8.302\mathrm e{-2}$ & 2.04\\
3 & 1280 & 5762 & $1.646\mathrm e{-1}$
  & $3.488\mathrm e{-5}$ & 2.90
  & $2.105\mathrm e{-2}$ & 2.02\\
\hline
\end{tabular}
\end{table}

Other Laplacian-dominated nonlinear surface models can be treated without changing the reconstructed surface Laplacian operator.  
One example is the surface analogue of the extended Fisher-Kolmogorov equation
\[
  \partial_tu
  +
  \gamma\deltag^2u
  -
  \eta\deltag u
  +
  \varepsilon^{-2}(u^3-u)
  =
  0.
\]
A genuinely vector-valued application with a componentwise
biharmonic leading term arises in fourth-order surface models for the
onset of cell blebbing; see
\cite{stinner2020fourth}.  The surface Cahn--Hilliard equation is
another possible application
\cite{du2011cahn,elliott2015cahn}, but its $H^{-1}$-gradient-flow
structure, mass conservation, and chemical-potential formulation
require additional considerations.  Reconstructed surface Hessians
for anisotropic fourth-order models are left for future work.

\section{Conclusions}
\label{sec:conclusions}

We have developed a reconstructed-Laplacian continuous/discontinuous Galerkin method for the surface biharmonic equation on parametric meshes.  
The reconstructed surface Laplacian operator consists of the broken Laplace-Beltrami operator and a local lifting of the conormal-flux jump.  
The method is stable for every fixed $\beta>0$. Hence no problem- or mesh-dependent lower threshold for
the stabilization parameter is required, although the stability constants may depend on $\beta$.

The analysis covers arbitrary finite element degree $k\ge2$ and
geometric degree $r\ge1$ and yields the leading estimates
\[
  \enorm{u^e-u_h}
  \lesssim h^{\min\{k-1,r\}},
  \qquad
  \norm{u^e-u_h}_{L^2(\Gammah)}
  \lesssim h^{\min\{q_k,r+1\}}.
\]
The experiments on benchmark examples confirm these rates for affine and curved meshes.  
The geometric diagnostics additionally provide computational evidence that the generated curved meshes satisfy the geometric assumptions used in the analysis and confirm a superconvergence in geometric quantities for the unit sphere and the particular computational meshes used when $r=2$. 

The same reconstructed Laplacian also applies directly to nonlinear
surface equations with a biharmonic principal part.  
We illustrate the method for the Swift-Hohenberg equation.  
Natural future directions include a complete nonlinear evolution analysis, mass-conservative and evolving-surface fourth-order problems, and tensor-valued reconstructions of the surface Hessian for anisotropic fourth-order problems on surfaces. 

\subsection*{Acknowledgment}
This work was supported by the NSFC grant No.12401512. 

AI tools were employed in a limited capacity to assist with code development and language editing. The author independently reviewed and validated all mathematical content, numerical results, and text, accepting full responsibility for the accuracy and integrity of the work.   
 
\appendix
\section{Technical lemmas}
\label{app:operator}
\begin{lemma}[Local strong Laplace-Beltrami comparison]
\label{lem:local-laplacian-perturbation}
Let $K\in\Kh$, let $K^l=p(K)$, and let $z\in H^2(K^l)$. Then
\begin{equation}\label{eq:pointwise-laplacian-perturbation}
  \left|
    \deltagh z^e-(\deltag z)^e
  \right|
  \le
  C\left(
    h_K^r|(\nablag z)^e|
    +
    h_K^{r+1}|(D_\Gamma^2z)^e|
  \right)
\end{equation}
almost everywhere on $K$.

Consequently, for every $v\in H^2(K)$,
\begin{equation}\label{eq:local-laplacian-perturbation}
  \norm{\deltagh v-(\deltag v^l)^e}_{L^2(K)}
  \le
  C\left(
    h_K^r\norm{\nablagh v}_{L^2(K)}
    +
    h_K^{r+1}
      \norm{D_{\Gammah}^2v}_{L^2(K)}
  \right).
\end{equation}
\end{lemma}

\begin{proof}
Throughout the proof, $C$ is independent of $K$ and $h$, and we
assume without loss of generality that $h_K\le 1$. We divide the proof
into four steps.

\medskip
\noindent
\emph{Step 1: Scaled coordinates and geometric estimates.}
Introduce the scaled reference triangle
\[
  \widetilde K:=h_K\widehat T
\]
and the parametrizations
\[
  \Phi_h(y):=F_K^r(y/h_K),
  \qquad
  \Phi_l(y):=p(\Phi_h(y)),
  \qquad y\in\widetilde K.
\]
Throughout the proof, $\nabla$ means the differentiation with respect to $y$. 
The shape-regularity assumptions \eqref{eq:parametric-shape-regularity} and \eqref{eq:curved-map-derivative-bounds} imply that 
\begin{equation}\label{eq:scaled-map-regularity}
  \norm{\nabla\Phi_h}_{L^\infty(\widetilde K)}
  +
  \norm{D^2\Phi_h}_{L^\infty(\widetilde K)}
  \le C.
\end{equation}
Similarly, $\nabla\Phi_h^T\nabla\Phi_h$ and $(\nabla\Phi_l)^T\nabla\Phi_l$ are of order $\mathcal{O}(1)$ and their eigenvalues are bounded above and below by positive constants independent of $K$ and $h$.

Let $d_h:=d\circ\Phi_h$.
The scaled distance condition \eqref{eq:scaled-distance-defect} gives
\begin{equation}\label{eq:scaled-distance-estimates}
  \norm{d_h}_{L^\infty(\widetilde K)}
  \le Ch_K^{r+1},
  \qquad
  \norm{\nabla d_h}_{L^\infty(\widetilde K)}
  \le Ch_K^r,
  \qquad
  \norm{D^2d_h}_{L^\infty(\widetilde K)}
  \le Ch_K^{r-1}.
\end{equation}

We shall use these estimates to show that the second-order coefficients of the two Laplace-Beltrami operators differ by $\mathcal{O}(h_K^{r+1})$, whereas their first-order coefficients differ by $\mathcal{O}(h_K^r)$.

\medskip
\noindent
\emph{Step 2: Estimates for the metric tensors.}
All exact geometric quantities in this step are evaluated at
$\Phi_h(y)\in\Gammah$. 
Recalling the relations \eqref{eq:closest-point}, it follows that $\nabla\Phi_l=(P-d_h\mathcal W)\nabla\Phi_h$. 
Using $P\mathcal W=\mathcal WP=\mathcal W$ and $\nabla d_h=\nabla(d\circ\Phi_h)=\nabla\Phi_h^T\vn$, 
we obtain the exact metric-defect identity
\begin{equation}\label{eq:metric-defect-identity}
  (\nabla\Phi_l)^T\nabla\Phi_l-(\nabla\Phi_h)^T\nabla\Phi_h
  =
  -\nabla d_h\otimes \nabla d_h
  -2d_h\nabla\Phi_h^T\mathcal W\nabla\Phi_h
  +d_h^2\nabla\Phi_h^T\mathcal W^2\nabla\Phi_h.
\end{equation}

Since $\mathcal W$ is uniformly bounded, \eqref{eq:scaled-map-regularity}-\eqref{eq:metric-defect-identity} yield
\begin{equation}\label{eq:metric-defect-zero}
  \norm{(\nabla\Phi_l)^T\nabla\Phi_l-(\nabla\Phi_h)^T\nabla\Phi_h}_{L^\infty(\widetilde K)}
    \le Ch_K^{r+1}.
\end{equation}

Differentiating \eqref{eq:metric-defect-identity} with respect to $y$ and using again the uniform bounds for $\nabla\Phi_h$, $D^2\Phi_h$, $\mathcal W$, and $\nabla\mathcal W$, we find
\begin{equation}\label{eq:metric-defect-first}
\begin{aligned}
  &\norm{\nabla\left((\nabla\Phi_l)^T\nabla\Phi_l-(\nabla\Phi_h)^T\nabla\Phi_h\right)}_{L^\infty(\widetilde K)}
  \\
  &\le
  C\Big(
    \norm{D^2d_h}_{L^\infty(\widetilde K)}
    \norm{\nabla d_h}_{L^\infty(\widetilde K)}
    +
    \norm{\nabla d_h}_{L^\infty(\widetilde K)}
  \\
  &\hspace{2.5em}
    +\norm{d_h}_{L^\infty(\widetilde K)}
    +
    \norm{d_h}_{L^\infty(\widetilde K)}\norm{\nabla d_h}_{L^\infty(\widetilde K)}
    +
    \norm{d_h}_{L^\infty(\widetilde K)}^2
  \Big)
  \\
  &\le Ch_K^r.
\end{aligned}
\end{equation}

Moreover, we define
\begin{equation*}
G_l:=(\nabla\Phi_l)^T\nabla\Phi_l,
\quad
G_h:=(\nabla\Phi_h)^T\nabla\Phi_h,
\quad
  (g_a^{ij}):=G_a^{-1},
  \quad
  q_a:=\log\det G_a,
  \quad a\in\{h,l\}.
\end{equation*}
Uniform positive definiteness of $G_h$ and $G_l$, together with the identity 
\begin{equation*}
G_l^{-1}-G_h^{-1}=-G_l^{-1}(G_l-G_h)G_h^{-1},
\end{equation*}
as well as  \eqref{eq:metric-defect-zero} and \eqref{eq:metric-defect-first}, implies 
\begin{align}
  \norm{G_l^{-1}-G_h^{-1}}_{L^\infty(\widetilde K)}
  &\le Ch_K^{r+1},
  \label{eq:inverse-metric-defect-zero}\\
  \norm{\nabla(G_l^{-1}-G_h^{-1})}_{L^\infty(\widetilde K)}
  &\le Ch_K^r.
  \label{eq:inverse-metric-defect-first}
\end{align}
Similarly, using $\nabla\log\det G=\operatorname{tr}(G^{-1}\nabla G),$
we obtain
\begin{align}
  \norm{q_l-q_h}_{L^\infty(\widetilde K)}
  &\le Ch_K^{r+1},
  \label{eq:determinant-defect-zero}\\
  \norm{\nabla(q_l-q_h)}_{L^\infty(\widetilde K)}
  &\le Ch_K^r.
  \label{eq:determinant-defect-first}
\end{align}

\medskip
\noindent
\emph{Step 3: Pointwise comparison of the two Laplace-Beltrami operators.}
We view $\Phi_{a}$ with $a\in\{h,l\}$ as parametrizations defined on $\widetilde K$ mapping to local charts on $\Gammah$ and $\Gamma$.  
For a scalar function $\widehat v$ on $\widetilde K$, the Laplace-Beltrami operators $\Delta_a$ ($a\in\{h,l\}$) are pull-backs of $\deltagh$ and $\deltag$, and the operator $\Delta_a$ is induced by the metric $G_a$ and has the coordinate representation
\begin{equation}\label{eq:coordinate-laplacians}
  \Delta_a\widehat v
  =
  g_a^{ij}\partial_{ij}\widehat v
  +b_a^j\partial_j\widehat v,
  \qquad
  b_a^j
  :=
  J_a^{-1}\partial_i(J_ag_a^{ij}),
  \qquad
  J_a:=(\det G_a)^{1/2}.
\end{equation}
Equivalently,
\[
  b_a^j
  =
  \partial_i g_a^{ij}
  +
  \frac12g_a^{ij}\partial_i q_a.
\]
Consequently,
\begin{align*}
  b_l^j-b_h^j
  &=
  \partial_i(g_l^{ij}-g_h^{ij})
  +
  \frac12(g_l^{ij}-g_h^{ij})\partial_iq_l
  +
  \frac12g_h^{ij}\partial_i(q_l-q_h).
\end{align*}
The individual metric coefficients and their first derivatives are
uniformly bounded. Therefore,
\eqref{eq:inverse-metric-defect-zero}-%
\eqref{eq:determinant-defect-first} give
\begin{equation}\label{eq:first-order-coefficient-defect}
  \norm{b_l-b_h}_{L^\infty(\widetilde K)}
  \le Ch_K^r.
\end{equation}

Now define the common pull-back
\[
  \widehat z
  :=
  z^e\circ\Phi_h
  =
  z\circ\Phi_l.
\]
Subtracting the two coordinate representations
\eqref{eq:coordinate-laplacians}, and using
\eqref{eq:inverse-metric-defect-zero} and
\eqref{eq:first-order-coefficient-defect}, yields
\begin{equation}\label{eq:coordinate-operator-defect}
  |(\Delta_h-\Delta_l)\widehat z|
  \le
  C\left(
    h_K^{r+1}|D^2\widehat z|
    +
    h_K^r|\nabla\widehat z|
  \right).
\end{equation}

The uniform regularity of $\Phi_l$ and its inverse, together with the chain rule, leads to almost everywhere
\[
  |\nabla\widehat z|
  \le
  C|(\nablag z)\circ\Phi_l|
\]
and
\[
  |D^2\widehat z|
  \le
  C\left(
    |(D_\Gamma^2z)\circ\Phi_l|
    +
    |(\nablag z)\circ\Phi_l|
  \right).
\]
Moreover,
\[
  (\Delta_h\widehat z)(y)
  =
  (\deltagh z^e)(\Phi_h(y)),
  \qquad
  (\Delta_l\widehat z)(y)
  =
  (\deltag z)(\Phi_l(y)).
\]
Substituting these identities into \eqref{eq:coordinate-operator-defect} and absorbing the term
$h_K^{r+1}|(\nablag z)\circ\Phi_l|$ into the $h_K^r$ term concludes the proof of \eqref{eq:pointwise-laplacian-perturbation}.

\medskip
\noindent
\emph{Step 4: The $L^2$ estimate.}
Let $v\in H^2(K)$ and set $z:=v^l\in H^2(K^l)$. Integrating
\eqref{eq:pointwise-laplacian-perturbation} over $K$ gives
\begin{align*}
  \norm{\deltagh v-(\deltag v^l)^e}_{L^2(K)}
  &\le
  C\Big(
    h_K^r\norm{(\nablag v^l)^e}_{L^2(K)}
    +
    h_K^{r+1}
      \norm{(D_\Gamma^2v^l)^e}_{L^2(K)}
  \Big).
\end{align*}
By norm equivalences in Lemma \ref{lem:norm-trace}, we have 
\[
  \norm{(\nablag v^l)^e}_{L^2(K)}
  \le
  C\norm{\nablagh v}_{L^2(K)}
\]
and
\[
  \norm{(D_\Gamma^2v^l)^e}_{L^2(K)}
  \le
  C\left(
    \norm{D_{\Gammah}^2v}_{L^2(K)}
    +
    \norm{\nablagh v}_{L^2(K)}
  \right).
\]
Consequently,
\begin{equation*}
  \norm{\deltagh v-(\deltag v^l)^e}_{L^2(K)}
  \le
  C\Big(
    h_K^r\norm{\nablagh v}_{L^2(K)}
    +
    h_K^{r+1}\norm{D_{\Gammah}^2v}_{L^2(K)}
    \Big).
\end{equation*}
This proves \eqref{eq:local-laplacian-perturbation}.
\end{proof}

\begin{lemma}[Smooth--smooth geometric consistency]
\label{lem:smooth-smooth-consistency}
For $z,\phi\in H^4(\Gamma)$,
\begin{equation}\label{eq:smooth-smooth-consistency}
  \left|
    a_h(z^e,\phi^e)-a(z,\phi)
  \right|
  \le
  Ch^{r+1}
  \norm{z}_{H^4(\Gamma)}
  \norm{\phi}_{H^4(\Gamma)}.
\end{equation}
\end{lemma}

\begin{proof}
Expanding \eqref{eq:expanded-form} for $z^e$ and $\phi^e$, and
inserting $(\deltag z)^e$ and $(\deltag\phi)^e$ into its two edge terms, gives
\begin{equation}\label{eq:smooth-smooth-decomposition}
\begin{aligned}
  a_h(z^e,\phi^e)-a(z,\phi)
  ={}&
  I_{0,h}+I_{\Pi,h}+I_{R,h}+I_{S,h},
\end{aligned}
\end{equation}
where
\begin{align}
  I_{0,h}
  :={}&
  (\deltagh z^e,\deltagh\phi^e)_{\Gamma_h}
  -\sum_{E\in\Eh}((\deltag z)^e,J_E(\phi^e))_E
  -\sum_{E\in\Eh}((\deltag\phi)^e,J_E(z^e))_E 
  \label{eq:smooth-principal-defect}\\
  {}&-(\deltag z,\deltag\phi)_\Gamma,
  \notag\\
  I_{\Pi,h}
  :={}&
  -\sum_{E\in\Eh}
  \left(
    \mean{
      \Pi_{Q_h}(\deltagh z^e-(\deltag z)^e)
      +(\Pi_{Q_h}-I)(\deltag z)^e
    },
    J_E(\phi^e)
  \right)_E
  \label{eq:smooth-projection-defect}\\
  &-
  \sum_{E\in\Eh}
  \left(
    \mean{
      \Pi_{Q_h}(\deltagh\phi^e-(\deltag\phi)^e)
      +(\Pi_{Q_h}-I)(\deltag\phi)^e
    },
    J_E(z^e)
  \right)_E,
 \notag\\
  I_{R,h}
  :={}&
  (R_hJ(z^e),R_hJ(\phi^e))_{\Gamma_h},
  \label{eq:smooth-lifting-defect}
  \\
  I_{S,h}
  :={}&
  \beta\sum_{E\in\Eh}h_E^{-1}
  (J_E(z^e),J_E(\phi^e))_E.
  \label{eq:smooth-penalty-defect}
\end{align}

We first estimate $I_{0,h}$.  We define 
\begin{equation*}
  q_z:=(\deltag z)^e,
  \quad
  q_\phi:=(\deltag\phi)^e,
\quad
  \varepsilon_z
  :=\deltagh z^e-(\deltag z)^e,
  \quad
  \varepsilon_\phi
  :=\deltagh\phi^e-(\deltag\phi)^e.
\end{equation*}
We integrate by parts elementwise to derive
\begin{align*}
  \sum_{E\in\Eh}(q_z,J_E(\phi^e))_E
  &=
  (\nablagh q_z,\nablagh\phi^e)_{\Gamma_h}
  +(q_z,q_\phi+\varepsilon_\phi)_{\Gamma_h},
  \\
  \sum_{E\in\Eh}(q_{\phi},J_E(z^e))_E
  &=
  (\nablagh q_\phi,\nablagh z^e)_{\Gamma_h}
  +(q_\phi,q_z+\varepsilon_z)_{\Gamma_h}.
\end{align*}
Then, 
\begin{align}
  &I_{0,h}
  =
  (\varepsilon_z,\varepsilon_\phi)_{\Gamma_h}
  -
  \Big[
    (\nablagh q_z,\nablagh\phi^e)_{\Gamma_h}
    -
    (\nablag\deltag z,\nablag\phi)_\Gamma
  \Big]
  \notag\\
  &-
  \Big[
    (\nablagh q_\phi,\nablagh z^e)_{\Gamma_h}
    -
    (\nablag\deltag\phi,\nablag z)_\Gamma
  \Big]
  -
  \Big[
    (q_z,q_\phi)_{\Gamma_h}
    -
    (\deltag z,\deltag\phi)_\Gamma
  \Big].
  \label{eq:smooth-principal-cancellation}
\end{align}
We have also used the exact surface integration-by-parts
\[
  (\nablag\deltag z,\nablag\phi)_\Gamma
  =
  -(\deltag z,\deltag\phi)_\Gamma=
  (\nablag\deltag\phi,\nablag z)_\Gamma.
\]

By \eqref{eq:laplacian-comparison-smooth}, the first term on the right-hand side of
\eqref{eq:smooth-principal-cancellation} satisfies
\begin{equation}\label{eq:smooth-operator-product}
  |(\varepsilon_z,\varepsilon_\phi)_{\Gamma_h}|
  \le
  Ch^{2r}
  \norm{z}_{H^2(\Gamma)}
  \norm{\phi}_{H^2(\Gamma)}
\end{equation}

The transformation \eqref{eq:gradient-transform} and the change of surface measure yield
\begin{equation*}
  (\nablagh q_z,\nablagh\phi^e)_{\Gamma_h}
  -
  (\nablag\deltag z,\nablag\phi)_\Gamma
 =
  \int_{\Gamma_h}
  (\nablag\deltag z)^e
  \cdot
  \bigl(B_hB_h^T-\mu_hP^e\bigr)
  (\nablag\phi)^e
  \,ds_h.
\end{equation*}
Therefore, by Assumption~\ref{lem:parametric-geometry},
\begin{align}
  \left|
    (\nablagh q_z,\nablagh\phi^e)_{\Gamma_h}
    -
    (\nablag\deltag z,\nablag\phi)_\Gamma
  \right|
  &\le
  Ch^{r+1}
  \norm{z}_{H^3(\Gamma)}
  \norm{\phi}_{H^1(\Gamma)}.
  \label{eq:smooth-gradient-defect-one}
\end{align}
The same argument gives
\begin{align}
  \left|
    (\nablagh q_\phi,\nablagh z^e)_{\Gamma_h}
    -
    (\nablag\deltag\phi,\nablag z)_\Gamma
  \right|
  &\le
  Ch^{r+1}
  \norm{\phi}_{H^3(\Gamma)}
  \norm{z}_{H^1(\Gamma)}.
  \label{eq:smooth-gradient-defect-two}
\end{align}
Finally, $(q_z,q_\phi)_{\Gamma_h}-(\deltag z,\deltag\phi)_\Gamma=\bigl((1-\mu_h)q_z,q_\phi\bigr)_{\Gamma_h}$,
and hence
\begin{equation}\label{eq:smooth-mass-defect}
  \left|
    (q_z,q_\phi)_{\Gamma_h}
    -
    (\deltag z,\deltag\phi)_\Gamma
  \right|
  \le
  Ch^{r+1}
  \norm{z}_{H^2(\Gamma)}
  \norm{\phi}_{H^2(\Gamma)}.
\end{equation}
Combining \eqref{eq:smooth-operator-product}--\eqref{eq:smooth-mass-defect}, we obtain 
\begin{equation}\label{eq:smooth-principal-final}
  |I_{0,h}|
  \le
  Ch^{r+1}
  \norm{z}_{H^3(\Gamma)}
  \norm{\phi}_{H^3(\Gamma)}.
\end{equation}

It remains to estimate $I_{\Pi,h}$. 
The scaled trace inequality, inverse inequality, the $L^2$-stability of $\Pi_{Q_h}$, and \eqref{eq:laplacian-comparison-smooth} imply
\begin{align}
 \left(\sum_{E\in\Eh}h_E\norm{\mean{\Pi_{Q_h}\varepsilon_z}}_{L^2(E)}^2\right)^{1/2}
  \le
  C\norm{\varepsilon_z}_{L^2(\Gamma_h)}
  \le
  Ch^r\norm{z}_{H^2(\Gamma)}.
  \label{eq:projected-operator-edge-z}
\end{align}
Since $q_z$ and $q_\phi$ possess at most two derivatives under the assumed $H^4$-regularity, the scaled trace inequality, inverse inequality and the projection approximation give,
with $\alpha_k:=\min\{k-1,2\}$,
\begin{align}
\left(\sum_{E\in\Eh}h_E\norm{\mean{(\Pi_{Q_h}-I)q_z}}_{L^2(E)}^2\right)^{1/2}
  &\le
  Ch^{\alpha_k}\norm{z}_{H^{2+\alpha_k}(\Gamma)}
  \le
  Ch^{\alpha_k}\norm{z}_{H^4(\Gamma)},
  \label{eq:smooth-projection-edge-z}
\end{align}
The same estimates hold for $\phi$ replacing $z$ in \eqref{eq:projected-operator-edge-z}-\eqref{eq:smooth-projection-edge-z}.   
Then using the weighted Cauchy-Schwarz in \eqref{eq:smooth-projection-defect}, 
\eqref{eq:projected-operator-edge-z}-\eqref{eq:smooth-projection-edge-z}, and \eqref{eq:smooth-geometric-jump-global}, we obtain 
\begin{equation}\label{eq:smooth-projection-final}
  |I_{\Pi,h}|
  \le
  C\left(h^{2r}+h^{r+\alpha_k}\right)
  \norm{z}_{H^4(\Gamma)}
  \norm{\phi}_{H^4(\Gamma)}.
\end{equation}

For the lifting-lifting and stabilization terms, lifting stability and \eqref{eq:smooth-geometric-jump-global} imply
\begin{align}
  |I_{R,h}|
  &\le
  Ch^{2r}
  \norm{z}_{H^2(\Gamma)}
  \norm{\phi}_{H^2(\Gamma)},\qquad 
  |I_{S,h}|
  \le
  C\beta h^{2r}
  \norm{z}_{H^2(\Gamma)}
  \norm{\phi}_{H^2(\Gamma)}.
  \label{eq:smooth-lifting-final}
\end{align}

Since $r\ge1$, $k\ge2$, we have $\alpha_k\ge1$, $2r\ge r+1$, $r+\alpha_k\ge r+1$.
Combining \eqref{eq:smooth-smooth-decomposition}, \eqref{eq:smooth-principal-final}, \eqref{eq:smooth-projection-final}, and \eqref{eq:smooth-lifting-final} proves \eqref{eq:smooth-smooth-consistency}.
\end{proof}

\bibliographystyle{unsrt}
\bibliography{ref_new}

\end{document}